\documentclass[11pt,reqno]{amsart}

\usepackage{graphicx}

\font\sml=cmr6

\newcommand{\K}{{\mathbf K}^\crit}

\newcommand{\RR}{\mathbb{R}}

\newcommand{\ZZ}{{\mathbb Z}}
\newcommand{\NN}{{\mathbb N}}
\newcommand{\R}{{\mathbb R}}

\newcommand{\N}{{\mathbb N}}

\newcommand{\E}{{\mathbf E}\,}

\newcommand{\half}{{\frac{1}{2}}}

\newcommand{\hcal}{\mathcal{H}}

\newcommand{\mcal}{\mathcal{M}}

\newcommand{\pcal}{\mathcal{P}}

\def    \half   {{\frac{1}{2}}}

\def    \R  {{\mathbb R}}

 \def   \half   {{\frac{1}{2}}}

\newtheorem{theo}{{\sc Theorem}}[section]
\newtheorem{cor}[theo]{{\sc Corollary}}

\newtheorem{remark}[theo]{{\sc Remark}}
\newtheorem{lem}[theo]{{\sc Lemma}}
\newtheorem{prop}[theo]{{\sc Proposition}}
\newtheorem{definition}[theo]{{\sc Definition}}

\newenvironment{defin-no-number}{\medskip\noindent{\it Definition:\/} }{\medskip}

\def\h#1{\hbox{#1}}

\def\o{\omega}

\def\K{K\"ahler }

\def\ra{\rightarrow}

\def\text{\textstyle}

\def\ra{\rightarrow}

\def\calP{\pcal}

\def\calM{\mcal}

\def\calH{\hcal} 

\def\calV{{\mathcal{V}}}

\def\diam{{\operatorname{diam}}}
\def\ginv{g^{-1}}

\def\half{\hbox{$\textstyle\frac12$}}
\def\quarter{\hbox{$\textstyle\frac14$}}
\def\eigth{\hbox{$\textstyle\frac18$}}
\def\sixtienth{\hbox{$\textstyle\frac1{16}$}}

\def\quartern{\hbox{$\textstyle\frac n4$}}

\def\sixtienthn{\hbox{$\textstyle\frac n{16}$}}
\def\nsixtienth{\hbox{$\textstyle\frac n{16}$}}

\font\sml=cmr5
\def\gNE{g_{\hbox{\sml N}}}
\def\gE{g_{\hbox{\sml E}}}
\def\dE{d_{\hbox{\sml E}}}
\def\LE{L_{\hbox{\sml E}}}
\def\gN{g_{\hbox{\sml N}}}
\def\dN{d_{\hbox{\sml N}}}

\def\NE#1#2{(#1,#2)_{\hbox{\sml N}}}
\def\N#1#2{(#1,#2)_{\hbox{\sml N}}}

\def\E#1#2{(#1,#2)_{\hbox{\sml E}}}

\def\NEsq#1{|#1|^2_{\hbox{\sml N}}}

\def\Npsquared#1#2{{(#1,#2)^2_p}}

\def\Nfsquared#1#2{{(#1,#2)^2_f}}
\def\Esq#1{|#1|^2_{\hbox{\sml E}}}
\def\KN{\hbox{$\bigcirc\mskip-14.8mu\wedge\;$}}

\def\Esml{\hbox{\sml E}}

\def\calMEbar{\overline{\calM_{\hbox{\sml E}}}}

\def\DiffM{\hbox{\rm Diff}(M)}

\def\KE{K\"ahler--Einstein }

\DeclareMathOperator{\Vol}{Vol}
\DeclareMathOperator{\tr}{tr}
\providecommand{\abs}[1]{\left\lvert #1 \right\rvert}
\providecommand{\integral}[4]{\int_{#1}^{#2} #3 \, #4}

\providecommand{\Mf}{\mathcal{M}_f}
\providecommand{\Mfhat}{\widehat{\Mf}}
\providecommand{\Mfplus}{\mathcal{M}_{f+}}
\providecommand{\Mfplushat}{\widehat{\Mfplus}}

\def\normN#1{|#1|_{\hbox{\sml N}}}

\title[Metrics on the space of Riemannian metrics
]
{Conformal deformations of the Ebin metric and a generalized Calabi metric
on the space of Riemannian metrics
}

\author{Brian Clarke}
\author{Yanir A. Rubinstein }

\address{Department of Mathematics, Stanford University, Stanford, CA 94305, USA}
\email{bfclarke@stanford.edu, yanir@member.ams.org}

\thanks{\hglue-10pt April 21, 2011.}

\begin{document}

\begin{abstract}
We consider geometries on the space of Riemannian
metrics conformally equivalent to the widely
studied Ebin $L^2$ metric.
Among these we characterize a distinguished metric that
can be regarded as a generalization of Calabi's
metric on the space of K\"ahler metrics to the space of
Riemannian metrics, and we study its geometry in detail. Unlike
the Ebin metric, the geodesic equation involves non-local terms,
and we solve it explicitly by using a constant of the motion.
We then determine its completion, which gives the first example
of a metric on the space of Riemannian metrics whose completion
is strictly smaller than that of the Ebin metric.
\end{abstract}

\maketitle

\bigskip
\section{Introduction}

Let $M$ be an $n$-dimensional compact closed manifold, and consider
the infinite-dimensional space $\calM$ of all smooth 
Riemannian metrics on $M$. The space $\calM$ is 
endowed with a natural $L^2$-type Riemannian structure, 
the Ebin metric \cite{E}, 
\begin{equation}\label{EbinMetDef}
\gE(h,k)|_g:= \E h k := \int_M \tr (\ginv h \ginv k) dV_g,
\end{equation}
where $g\in\calM$, $h,k\in T_g\calM$, $T_g\calM$ may be identified
with the space $\Gamma(S^2 T^* M)$ of smooth symmetric (0,2)-tensor
fields on $M$, and $g^{-1} h$ represents the $(1,1)$-tensor dual to
$h$ with respect to $g$.  This metric has received much attention
since being introduced in the 1960s, see, e.g., \cite{FG,GM,Cl1,Cl5},
and has found various applications, for example in the Weil--Petersson
geometry of moduli spaces of Riemann surfaces \cite{FT,T}
and in the study of the moduli space $\calM/\DiffM$ of Riemannian
structures (e.g.,~\cite{E,F,Bo}).  A related pseudo-Riemannian metric,
the DeWitt metric \cite{D,P}, has been used in the Hamiltonian
formulation of general relativity.

Recently, the metric completion of $\calMEbar$ of $(\calM,\gE)$ has been determined \cite{Cl3},
and it was shown by means of examples that convergence in $\calMEbar$
is too weak to control any geometric quantities or to imply
geometric convergence of any sort (e.g., Gromov--Hausdorff convergence) \cite{Cl4}.
Therefore, it seems natural to look for other metrics on $\calM$
with the property that their metric completions are stricly contained
in $\calMEbar$. In other words, metrics for which certain types
of degenerations are excluded along convergent sequences.
One purpose of this article is to take a first step in this direction
by studying conformal deformations of
the Ebin metric in the search for metrics with this and 
other distinguished properties. 

Our first observation (Proposition
\ref{VFactorProp}) is that there is a distinguished metric
in the conformal class characterized by the property that
the tautological vector field $X|_g=g$ on $\calM$ is parallel.
This metric, which we call the {\it generalized Calabi metric} (or sometimes
the {\it normalized Ebin metric}), is given 
by
$$
\gN:=\frac1{V_g}\gE, \quad g\in \calM,
$$
where $V_g:=\Vol(M,g)$ is the volume function on $\calM$.
We then restrict attention to conformal factors that
depend on the volume, i.e., metrics on $\calM$ of the form
$e^{2f(V_g)}\gE$, with $f$ a smooth function on $\RR_{>0}$,
and mostly to the metrics $g_p:=\gE/V^p$, which serve as the basic
models within this family, as they capture the possible
degenerations of manifolds in terms of either volume collapse
or blow-up. By studying this family of metrics, we then show
that $\gN$ has the smallest metric completion (Theorem \ref{dpCpl}), 
and in particular one that is smaller than that of the Ebin metric.
This provides the first example of an $L^2$-type metric on $\calM$ 
whose metric completion is strictly smaller than that of the
Ebin metric.

An additional motivation for introducting $\gN$ comes from
the study of the subspace of \K metrics $\calH\subset\calM$ 
in a fixed \K class (when $M$ admits a \K structure). 
In our previous work \cite{CR}, we studied the
intrinsic and extrinsic geometry of $\calH$ in $\calM$.
We observed that the Ebin metric induces
the so-called Calabi geometry on $\calH$, and that this embedding is
as far from being totally geodesic as possible.
It then seems natural to ask whether there exists a metric
on $\calM$ that still induces the Calabi geometry on $\calH$
but with the property that $\calH$ is totally geodesic. As before, 
it is natural to restrict to conformal deformations
depending on the volume, this time since the volume is an invariant
of the \K class, and so any such metric will induce
the Calabi geometry on $\calH$. We then show that to the
extent possible, $\gN$ is the unique metric with the 
aforementioned property.
In particular, $\calH$ is totally geodesic in the case that $M$ is a
Riemann surface.
In general $\mathcal{H}$ is not totally geodesic, but by the 
Calabi--Yau Theorem it is isometric
to the ``Riemannian \K spaces" $\calP g\cap\calM_v$,
consisting of metrics of fixed volume in a fixed conformal class, 
which are totally geodesic in $(\calM,\gN)$ (Corollary
\ref{HTotallyGeodCor}).

One further possible application of the metric $\gN$ is to
the Ricci flow. Recently, we showed that in the \K setting there is a
connection between the existence of Einstein 
metrics, the smooth convergence of the normalized Ricci
flow, and the metric geometry of $(\calM,\gE)$. Namely,
a \KE metric exists on a Fano manifold if and only if
the K\"ahler--Ricci flow converges in the metric completion of $(\calM,\gE)$,
and in particular if and only if the flow path has finite length
\cite{CR}. 
It would be very interesting to find analogous results
for other classes of Riemannian manifolds, perhaps ones
for which the singularities of the Ricci flow can be understood fairly well. 
In studying this problem, it might prove useful to use the metric $\gN$,
for which the the submanifold $\calM_v\subset\calM$ of metrics of fixed 
volume $v$---which is preserved by the normalized Ricci flow---is also
totally geodesic (Corollary \ref{MvTotGeod}).

Motivated by these and possible other applications of the 
metric $\gN$ to geometric problems,
we thus study the 
geometry of $(\calM,\gN)$ in detail.  Under the conformal change, the geodesic equation becomes
substantially more difficult since it contains non-local terms
involving integation over the whole manifold. The solution is obtained in several
steps, building upon the work of Freed--Groisser for $\gE$ \cite{FG}.
A key extra ingredient here is an invariant of the $\gN$-geodesic
flow, or a `constant of motion' (Corollary \ref{HessianLogVCor}).
The solution of the geodesic equation (Theorem \ref{gNGeod})
gives a precise sense to how geodesics in $(\calM,\gN)$
generalize those discovered by Calabi \cite{C1,C2} for
the subspace of \K metrics, which in turn bear several similarities
with constrained geodesics of the Wasserstein metric in optimal
transportation \cite{CG} (cf. \cite{CR}).
We also compute the curvature
of $\gN$ and compare it to that of the metrics $g_p$ 
(\S\ref{CurvatureSubSec}).

Finally, it should be noted that ``weighted" $L^2$ type metrics
were also studied by several authors on the space of simple closed
curves in $\RR^2$ (see \cite{MM1,MM3,Sh} and references therein), and this can also be seen as another
motivation for our study. Moreover, very recently, while the present
article was being prepared, Bauer--Harms--Michor \cite{BHM} have written
down the geodesic equation for metrics conformal to $\gE$ on
$\calM$, as well as much more general Sobolev-type metrics and metrics
weighted by the scalar curvature function.
Their main result is
that for some of these metrics the exponential
mapping is a local diffeomorphism.
In this article, we go into greater depth for a smaller class of
metrics by solving the geodesic equation, computing the curvature,
estimating the distance function, and determining the metric completion.

The article is organized as follows.  In Section \ref{PrelimSec}, we
briefly review the relevant preliminaries about $\mathcal{M}$.  In
Section \ref{Sec:ConfDefEbin}, we discuss general conformal changes,
mostly focusing on those involving functions of the volume. For
the model metrics $g_p$ we compute the curvature as well as find 
an invariant (or `constant of motion') of the geodesic flow.
Section \ref{GeomgN} contains the solution of the initial value problem for
$\gN$-geodesics, making use of the invariant of the geodesic flow.
In Section \ref{distance-functions}, we study the distance functions
of $g_p$ and determine their metric completions. Some of the technical
facts needed in this analysis are proven in an Appendix.  Section
\ref{sec:Remarks} concludes with some further remarks and a few open questions.

\subsection*{Acknowledgements}
This material is based upon work supported in part by NSF grants 
DMS-0902674, 0802923.  B.C.~thanks X.~Dai and
G.~Wei for interesting discussions related to weighted $L^2$
metrics on $\mathcal{M}$.

\section{Preliminaries}
\label{PrelimSec}

Since the preliminaries relevant to our results are covered in detail
in \cite{E,FG,GM}, we will simply briefly summarize what we need in
this section.

The manifold of metrics, $\mathcal{M}$, is easily seen to be an open
cone in the Fr\'echet space $\Gamma(S^2 T^* M)$ of smooth, symmetric
$(0, 2)$-tensors on the finite-dimensional, compact manifold $M$.  As
such, it is endowed with the structure of a Fr\'echet
manifold, and its tangent space at $g \in \mathcal{M}$ is canonically
identified with $\Gamma(S^2 T^* M)$.

The Ebin metric, defined in \eqref{EbinMetDef}, is a smooth Riemannian
metric.  It is, however, a weak metric, meaning that the tangent
spaces of $\mathcal{M}$ are incomplete with respect to the scalar
product induced on them by the Ebin metric.  For weak Riemannian
metrics, the existence of the Levi--Civita connection is not
guaranteed by any general results.  Nevertheless, the Ebin metric has
a Levi--Civita connection which can be directly computed.  Geodesics
and curvature may also be directly computed. The Riemannian curvature
of $(\mathcal{M}, \gE)$ is nonpositive, and the exponential mapping at
any point $g \in \mathcal{M}$ is a real-analytic diffeomorphism from
an open neighborhood of zero in $T_g \mathcal{M}$ to an open
neighborhood of $g$ in $\mathcal{M}$.  (Both of these neighborhoods
are taken in the $C^\infty$ topology.)

With respect to $\gE$, we may orthogonally decompose the tangent space
$T_g \mathcal{M}$ into the subspaces of traceless (satisfying
$\tr(g^{-1} h) = 0$) and pure-trace (satisfying $h = \rho g$ for some
$\rho \in C^\infty(M)$) tensor fields.
Corresponding to this decomposition is a product manifold structure
for $\mathcal{M}$.  Denote by $\mathcal{V}$ the space of all smooth,
positive volume forms on $M$; it is an open cone in $\Omega^n(M)$, the
space of smooth top-degree forms.  For any $g \in \mathcal{M}$ we
denote by $dV_g$ its induced volume form.  Then for any $\mu \in
\mathcal{V}$, with $\mathcal{M}_\mu := \{ g \in \mathcal{M} \, : \,
dV_g = \mu \} \subset \mathcal{M}$, there is a diffeomorphism
\begin{equation}
\label{imuEq}
  i_\mu : \mathcal{V} \times \mathcal{M}_\mu \rightarrow \mathcal{M}
  \qquad i_\mu(\nu, h) =
  ({\nu}/{\mu})^{2/n}
  h.
\end{equation}
That is, $i_\mu$ maps $(\nu, h)$ to the unique metric conformal to $h$
with volume form $\nu$.  Thus $\mathcal{M} \cong \mathcal{V} \times
\mathcal{M}_\mu$, and one sees that $(i_\mu)_*(T \mathcal{V})$ is the
subbundle of $T \mathcal{M}$ consisting of pure-trace tensor fields,
while a tangent space to the submanifold $\mathcal{M}_\mu$ is
identified with the subspace of traceless tensor fields.

An identity that will be repeatedly used below is that
the differential of the map $g
\mapsto dV_g$ is $h \mapsto \frac{1}{2} \tr(g^{-1} h) dV_g$.
Therefore, if we denote by $V = V_g := \int_M dV_g$ the volume
function on $\calM$, then 
the differential of $g \mapsto V_g$ is $h \mapsto \frac{1}{2} \E g h$.

\section{Conformal deformations of the Ebin metric}
\label{Sec:ConfDefEbin}

Let $f:\calM\ra\RR$ be a twice continuously differentiable function, and
consider the metric on $\calM$,
\begin{equation}
g_f(h,k)|_g:=e^{2f(g)}\gE(h,k)|_g=e^{2f(g)}\int_M \tr (\ginv h \ginv k) dV_g,\quad h,k\in T_g\calM,
\end{equation}
conformal to the Ebin metric. 
The purpose of this section is to characterize two metrics in
the conformal class of $\gE$. One metric, the generalized Calabi
metric $\gN=\gE/V$, is characterized
by its Levi-Civita connection (Proposition \ref{VFactorProp}), 
and the other, the second Ebin metric $g_2=\gE/V^2$, 
by its curvature tensor (Proposition \ref{MiracleProp}). 
We then restrict to the model metrics
\begin{equation}\label{gpDef}
  g_p := \frac{1}{V_g^p} \gE, \quad g\in\calM
\end{equation}
for some integer $p$.
We find invariants for their geodesic flows that
will be important later in integrating the geodesic equation
(Corollary \ref{HessianLogVCor} and Lemma \ref{VolumePowerLinearLemma}),
compute their curvature (\S\ref{CurvatureSubSec}), and describe
a natural duality map on $\calM$ that is a conformal isometry
between $g_p$ and $g_{2-p}$ (Proposition \ref{IsometryProp})
and that also conformally relates their curvature tensors.

\subsection{Conformal deformations and the Levi-Civita connection}
\label{LCConfSubSec}

Our first observation is a characterization of the {\it generalized
Calabi metric}
$$
\gN:= g_1 = \gE/V_g, \quad g\in\calM.
$$

\begin{prop}
\label{VFactorProp}
Let $f:\calM\ra\RR$ be a smooth function. Let $\nabla^f$ denote the
Levi-Civita connection of $e^{2f}\gE$, and suppose that $\nabla^f g=0$,
where $g$ denotes the tautological vector field $g \mapsto g$ on
$\calM$.  Then $f(g)=-\half\log V_g+C$ for some constant $C$.
\end{prop}

\begin{proof}
First, note that $\nabla^{\gE}g=\frac n4\delta$, where $\delta$ is the
Kronecker tensor. 
This follows easily from the formula for the Levi--Civita connection
of $\gE$ \cite[(4.1)]{E},
\begin{equation*}
  \left. \nabla^{\gE}_h k \right|_g = D_h k - \frac{1}{2} (h g^{-1} k + k g^{-1} h) +
  \frac{1}{4}
  \left(
    \tr(g^{-1} k) h + \tr(g^{-1} h) k - \tr(g^{-1} h g^{-1} k) g
  \right),
\end{equation*}
where $h$ and $k$ are any vector fields on $\mathcal{M}$ and
$\left. D_h k \right|_g = \left. \frac{d}{dt} \right|_{t=0} k(g + t
h)$.

Next, recall that \cite[p. 58]{B}, 
\begin{equation}
\label{LCConformalChangeEq}
\nabla^f_hk
=
\nabla^{\gE}_hk+(\nabla_hf)k+(\nabla_kf)h-\E hk\nabla^{\gE} f,
\end{equation}
so $\nabla g=0$ is equivalent to
$$
0=\frac n4 h+df(h)g+df(g)h-\E hg\nabla^{\gE} f, \quad \h{for all\ } h.
$$
Plugging in $h=g$ shows that $\nabla^{\gE}f$ is proportional to $g$;
and by inspecting the equation again then necessarily 
$-\frac n4=df(g)=D_g f=\nabla_g^{\gE}f$ and $dF(h)g=\E hg\nabla^{\gE}f$. Combining these two equations
yields $\nabla^{\gE} f=-\frac1{4V}g$ and substituting this
back into the second equation yields $df(h)=-\quarter\NE gh$.
Now consider a path $\{g(t)\}$. Then
$$
\frac d{dt} f(g(t))
=-\frac1{4V}\int_M \tr(g(t)^{-1}g_t)dV_{g(t)}
=-\frac12\frac d{dt}\log V_{g(t)},
$$
hence $f(g)=-\half\log V_g+C$ (as $\calM$ is path connected),
as desired.
\end{proof}

Since, by the proof above, $g$ is the gradient vector field of $2\log
V$ with respect to $\gN$, we
have the following corollary, which will prove
crucial in integrating the geodesic equation for $\gN$
in \S\ref{GeomgN}.

\begin{cor} 
\label{HessianLogVCor}
The Hessian of $\log V$ satisfies
$$
\nabla^{\gN}\,d\,{\log V}=0.
$$
In particular, $\log V$ is linear and $V$ is either strictly monotone
or constant along $\gN$-geodesics.
\end{cor}

By the above corollary, if $g(t)$ is a $\gN$-geodesic and
$(V_{g(t)})_t(0) = 0$, then $V_{g(t)}$ is constant.  This gives the
following fact.

\begin{cor}\label{MvTotGeod}
  For any $v \in \RR_+$, the submanifold $\mathcal{M}_v := \{ g \in
  \mathcal{M} \, : \, V_g = v \}$ is totally geodesic in
  $(\mathcal{M}, \gN)$.
\end{cor}

As Corollary \ref{VolumeGpGeodCor} below will imply, the above
statement is true for $g_p$ only when $p = 1$.  In particular, it is
false for the Ebin metric.

By using the Koszul formula 
(or else by using (\ref{LCConformalChangeEq}) and the known
expression for $\nabla^{\gE}$),
one can directly compute the Levi-Civita connection of $\gNE$
for constant vector fields $h,k$ to be
\begin{equation}
\label
{gNLCConnectionEq}
\begin{aligned}
 \nabla^{\gN}_hk|_g
& =
\quarter \gNE(k,h)g-\quarter\tr(\ginv h\ginv k)g
-\half h\ginv k-\half k\ginv h
\cr
&\quad
+\quarter\tr(\ginv h)k
+\quarter\tr(\ginv k)h
-\quarter \gNE(h,g)k
-\quarter \gNE(k,g)h.
\end{aligned}
\end{equation}
It is torsion free (symmetric in $h$ and $k$), and
one checks directly that it is metric compatible, 
hence it is the Levi-Civita connection.

It is well-known that along $\gE$-geodesics the volume is quadratic \cite{FG}.
This is explained by the following Lemma and Corollary, which are in a similar vein to 
Corollary \ref{HessianLogVCor}.

\begin{lem}
\label{VolumePowerLinearLemma}
We have $$
\nabla^{g_p}dV^{1-p}=\frac n8(1-p)^2g_p.
$$
\end{lem}
\begin{proof}
By (\ref{LCConformalChangeEq}) and (\ref{gNLCConnectionEq})
for constant vector fields $h,k$,
$$
\begin{aligned}
 \nabla^{g_p}_hk|_g
& =
\frac{p}{4}  \N kh g-\quarter\tr(\ginv h\ginv k)g
-\half h\ginv k-\half k\ginv h
\cr
&\quad
+\quarter\tr(\ginv h)k
+\quarter\tr(\ginv k)h
-\frac{p}{4}\N hgk
-\frac{p^2}{4} \N kgh.
\end{aligned}
$$
Thus,
$$
\begin{aligned}
\nabla^{g_p}dV^{1-p}(h,k)
& =
\half\nabla^{g_p}_h(1-p)V^{-p}\E gk-\half(1-p)V^{-p}\E g{\nabla^{g_p}_hk}
\cr
& =
-\frac{p(1-p)}{4V^{p+1}}\E gh\E gk-\frac{1-p}{2V^p}\E hk
+\frac{1-p}{4V^p}\E g{\tr(\ginv h)k}
\cr
&
\quad -\frac{np(1-p)}{8V^p}\E hk+\frac{n(1-p)}{8V^p}\E hk+\frac{1-p}{2V^p}\E hk
\cr
&
\quad -\frac{1-p}{4V^p}\E g{\tr(\ginv h)k}+\frac{p(1-p)}{4V^{p+1}}\E gh\E gk
\cr
& = \frac n8 (1-p)^2 g_p(h,k).
\end{aligned}
$$

\end{proof}

\begin{cor}
\label{VolumeGpGeodCor}
Along unit-speed $g_p$-geodesics $g(t)$, $V^{1-p}$ grows quadratically ($p\ne 1$),
$$
V^{1-p}(t)
=
\sixtienthn (1-p)^2 t^2+\frac{1-p}2a_0t+V^{1-p}(0),
$$
where
$$
a_0:=
\frac{1}{V^p_{g_0}}\int_M f(0)dV_{g_0},\qquad f(t):=\tr(\ginv g_t).
$$

In particular, $V_{g(t)}$ converges to $0$ (if $p<1$) or to $\infty$
(if $p>1$) in finite time precisely along constant conformal
directions, i.e., if $g_t(0)=\lambda g(0)$ for $\lambda\in\RR$ and
$\lambda$ negative (if $p<1$) or positive (if $p>1$).  Also, along a
unit-speed $g_p$-geodesic,
\begin{equation}
\label{ConstantMotionEq}
\frac{d}{dt}\Big(\frac1{V^p}\int_M f\, dV_g\Big)=\frac n{4}(1-p).
\end{equation}

\end{cor}

The last equation provides an integral for the geodesic flow
of $g_p$ which allows solving the geodesic equation explicitly,
in the spirit of the work of Freed--Groisser. This will be 
carried out for $\gN$ using Corollary \ref{HessianLogVCor} in
\S\ref{GeomgN}.

\subsection{Conformal deformations and curvature}
\label{CurvatureSubSec}

Our main purpose in this subsection is to study which conformal
deformations of $\gE$ still have non-positive curvature. The next
result shows that there is precisely one metric of the form $\gE/V^p$,
besides $\gE$ itself, whose curvature is nonpositive---it is also the
unique such metric with curvature conformal to that of $\gE$---and
this characterizes the {\it second Ebin metric}
$$
g_2 = \gE/V^2,
$$
among all conformal deformations that depend on the volume.

\begin{prop}
\label{MiracleProp}
The curvature of $g_p$ is nonpositive if and only
if $p=0$ or $2$. Moreover, the curvature of $g_2=\gE/V^2$ is 
conformal to the curvature of $\gE$,
$$
R^{g_2}=R^{\gE}/V^2,
$$
and this property characterizes $g_2$, up to scaling, among all
conformal deformations $e^{2f}\gE$ with $f:\calM\ra\RR$ a smooth
function depending only on $V_g$.
\end{prop}

In the proof we make use of the following computation:

\begin{prop}
\label{CurvgpProp}
  The curvature tensor of $g_p$ is given by
  \begin{equation}
\label{CurvN2Eq}
    R^{g_p} = 
\frac{1}{V^p} R^{\gE} 
+
\frac{2p-p^2}{16}g_p \KN 
     \left( V^{2p-2}g^{\flat_p} \otimes g^{\flat_p} - \frac{n}{2} V^{p-1}g_p \right),
  \end{equation}
where $g^{\flat_p}$ is the 1-form dual to the tautological vector field $g$
with respect to $g_p$ and $\KN$ denotes the Kulkarni--Nomizu product.

Let $h,k$ be tangent vectors that are orthonormal with respect to
$g_p$.  The sectional curvature of the plane $\RR\{h,k\}$ is
\begin{equation}
\label{SecCurvNpEq}
\h{\rm sec}^{g_p}(h,k)
=
\frac1{V^p}\h{\rm sec}^{\gE}(h,k)
-\frac{2p-p^2}{16}
\big(V^{2p-2}\Npsquared gk+ V^{2p-2}\Npsquared gh-nV^{p-1} \big).
\end{equation}
\end{prop}

\begin{proof}[Proof of Proposition \ref{CurvgpProp}]
The formula for the curvature under the conformal 
change $\gE \mapsto e^{2f}\gE$ is \cite[p. 58]{B},
\begin{equation}
\label{ConformalCurvatureEq}
\begin{aligned}
R &=\frac1V 
\big(
R^{\gE}
+\gE\KN(\nabla^{\gE} df-df\otimes df+\half|df|_{\hbox{\sml E}}^2\,\gE)
\big),
\end{aligned}
\end{equation}
and this applies in infinite dimensions as can be verified from its proof.  (Note that our convention is
$R(X,Y)Z=(\nabla_X\nabla_Y-\nabla_Y\nabla_X - \nabla_{[X,Y]})Z$,
the opposite of Besse's.)
Let $f(g):=-\frac12 \log V_g$ and 
$f_p(g):= p f(g)=-\frac p2\log V_g$. 
Assume first that $p=1$.
We claim that
\begin{equation}
\label{nablagEEq}
\nabla^{\gE} d f=\frac{1}{8}\NE g{\,\cdot\,}\NE g{\,\cdot\,}-\frac n{16}\gNE
=2df\otimes df-\frac n{16}\gNE.
\end{equation}
To see this, compute using the formula $\nabla^{\gE}_h g = \frac{n}{4}
h$ (cf.~the proof of Proposition \ref{VFactorProp}) and the metric
property of $\nabla^{\gE}$
to deduce
\begin{equation}\label{fHess}
\begin{aligned}
\nabla^{\gE} df(h,k) &
=(\nabla^{\gE}_hdf)(k)
=\nabla^{\gE}_h(df(k))-df(\nabla^{\gE}_hk)
\cr
&=
-\quarter\nabla^{\gE}_h
\left(
  \frac{1}{V} \E g k
\right)+\quarter\NE g{\nabla^{\gE}_hk}
\cr
&=
\eigth\NE gh\NE gk-\nsixtienth\NE hk-\quarter\NE g{\nabla^{\gE}_hk}+\quarter\NE g{\nabla^{\gE}_hk}
\cr
&=
\eigth\NE gh\NE gk-\nsixtienth\NE hk.
\end{aligned}
\end{equation}
Second, note that $df=-\quarter g^{\flat_1}$, 
$
\NEsq{df}=\NEsq{\nabla f}=\sixtienth\NEsq g=\nsixtienth,
$
and
$\NEsq{df}\gNE=\Esq{df}\gE$. Thus,
  \begin{equation}
\label{CurvgNEq}
    R^{\gN} = \frac{1}{V} R^{\gE} +\frac1{16}\gNE \KN 
     \left( g^{\flat_1} \otimes g^{\flat_1} - \frac{n}{2} \gN \right),
  \end{equation}

To conclude the proof, note now that
$
\nabla^{\gE} d f_p
=
p\nabla^{\gE} d f
=
2pdf\otimes df-\frac {pn}{16}\gNE,
$
$df_p\otimes df_p=p^2df\otimes df,$ and
$\half|df_p|_{\hbox{\sml E}}^2\,\gE
=\frac{p^2}{2}|df|_{\hbox{\sml E}}^2\,\gE=\frac{p^2 n}{16}\gN$.
From (\ref{ConformalCurvatureEq}) we thus obtain
$$
R^{g_p}=
\frac{1}{V^p} R^{\gE} 
+
(2p-p^2)g_p \KN 
\left(df \otimes df - \frac{n}{32} V^{p-1}g_p \right).
$$
Since $df=-\quarter\N g\cdot=-\quarter g^{\flat_1}=-\quarter V^{p-1}g^{\flat_p}$,
(\ref{CurvN2Eq}) follows.

Next, recall that 
$$G\KN H(a,b,c,d)
=
G(a,c)H(b,d)+G(b,d)H(a,c)
-G(a,d)H(b,c)-G(b,c)H(a,d).
$$
So if $h$ and $k$ are $g_p$-orthonormal,
$
g_p \KN g_p(h,k,k,h)=2(h, k)_p^2-2|h|_p^2 |k|_p^2=-2,
$
and
$$
\begin{aligned}
g_p \KN (g^{\flat_p}\otimes g^{\flat_p})(h,k,k,h) &=
2(h, k)_p (g^{\flat_p}\otimes g^{\flat_p})(h,k)
\cr
&\qquad
{}-(h, h)_p (g^{\flat_p}\otimes g^{\flat_p})(k,k)
-(k, k)_p (g^{\flat_p}\otimes g^{\flat_p})(h,h)
\cr
&
=
-(g, k)_p^2
-(g, h)_p^2.
\end{aligned}
$$
By definition, $\h{\rm sec}^{g_p}(h,k)=g_p(R^{g_p}(h,k)k,h)$, 
and so (\ref{SecCurvNpEq}) follows. 
\end{proof}

\begin{proof}[Proof of Proposition \ref{MiracleProp}]
Let $f=f(V_g)$ be a smooth function on $\calM$. 
Then $df=f'dV=\half f'\E g{\,\cdot\,}$, or
$\nabla^{\gE} f=\half f'g$ and 
$\half\Esq{\nabla^{\gE} f}\gE=\frac n8(f')^2 V\gE$,
while $df\otimes df=\quarter (f')^2g^{\flat_{\Esml}}\otimes g^{\flat_{\Esml}}$.
A computation similar to \eqref{fHess} gives
$\nabla^{\gE}df
=
\quarter f''g^{\flat_{\Esml}}\otimes g^{\flat_{\Esml}}
+\frac n8f'\gE
$. So 
$$
R^{g_f}=e^{2f}R^{\gE}
+\quarter e^{2f}\gE\KN
\left(
(f''-(f')^2)
g^{\flat_{\Esml}}\otimes g^{\flat_{\Esml}}
+\frac n2(f'+V(f')^2)\gE
\right),
$$
and analogously to the proof of \eqref{SecCurvNpEq}, we may compute
$$
\h{\rm sec}^{g_f}(h,k)
=
e^{2f}\h{\rm sec}^{\gE}(h,k)
-
\frac{f''-(f')^2}{4e^{4f}}
(\Nfsquared gk+\Nfsquared gh)
-\frac n{4e^{2f}}(f'+V(f')^2).
$$
Suppose now that $\h{\rm sec}^{g_f}=e^{2f}\h{\rm sec}^{\gE}$.  Then
considering directions $h,k$ tangent to $\calM_\mu$ gives that
$f'+V(f')^2=0$,
from which it follows that either
$f=-\log V+C$, i.e., $g_f=e^{2C}\gE/V^2$, or else $f'=0$, i.e.,
$g_f=e^{2C}\gE$.
\end{proof}

\subsection{Conformal transformations and a duality map}

By Proposition \ref{CurvgpProp}, $g_{2-p}$ and $g_p$ have 
the same curvature tensor, up to a conformal factor. 
Here we observe that there is 
also a conformal diffeomorphism $F:\calM\ra\calM$ that 
relates these two metrics, so they
are in fact isometric, and in this sense $g_2$ 
does not provide a new geometry compared to $\gE$.

Consider the map $F:\calM\ra\calM$ defined
by $F(g):=V^{q} g$.
Let $h\in T\calM$ be a constant vector field.
Then
$$
\begin{aligned}
dF(h)
&=
\frac d{dt}\Big|_{t=0}(g+th)V_{g+th}^{q}
=V^{q}h+\half q\E gh V^{q-1}g.
\end{aligned}
$$
Hence, a careful computation shows that
$$
\begin{aligned}
g_p(dF(h),dF(k))|_{F(g)}
& =
g_{p+\frac n2 q(p-1)}(h,k)+
\big(\quartern q^2+q\big)
V^{(1-p)(1+\frac n2 q)-2}\E gh\E gk.
\end{aligned}
$$

To summarize, we have:

\begin{prop}
\label{IsometryProp}
The diffeomorphism $F(g)=V^{-\frac4n} g$ of $\calM$
is an isometry between $(\calM,g_{2-p})$ and
$(\calM,g_p)$, and we have $V_{F(g)} = V_g^{-1}$ and $F^{-1} = F$.
In particular, $(\calM,g_2)$ and $(\calM,\gE)$
are isometric.
\end{prop}

It is interesting to note that using this result one obtains rather
effortlessly the solution of the geodesic equation for $g_2$, building
on the much simpler one for $\gE$ (\cite[Thm.~3.2]{GM},
\cite[Thm.~2.3]{FG}).
In fact, a direct
solution of the $g_2$ geodesic equation using the fact that the
inverse of the volume is quadratic (Corollary \ref{VolumeGpGeodCor})
is substantially more involved.

\begin{remark}
{\rm
If $\phi$ is a positive differentiable function,
and $F(g):=\phi(V_g) g$, then 
$$
(dF(h),dF(k))_p|_{F(g)}=\frac{\phi^{n/2}(V_{F(g)})}{V^{p}}\E hk|_g
+
V_{F(g)} \E gh\E gk \frac{\phi^{n/2} \phi'}{\phi}
\Big(\frac{n\phi'}{4\phi}V+1\Big).
$$
Hence, the only such map $F$ that is an isometry between
$\gE$ and any $g_p$ is given by Proposition \ref{IsometryProp}.

}
\end{remark}

\section{Geometry of the generalized Calabi metric}
\label{GeomgN}

In this section, we study the geometry of the metric $\gN$
in more detail.  In the first subsection we solve its
geodesic equation for any given initial data.  
In the second
subsection, we compute the sectional curvature, and
examine the extrinsic geometry of certain submanifolds in the spirit of \cite{CR},
showing that the Riemannian analogues of the space of \K metrics are 
totally geodesic. These spaces are naturally isometric (via the Calabi--Yau
Theorem) to the usual spaces of \K metrics.
These facts, together with the explicit formula for geodesics, 
give a precise meaning to the statement that $\gN$ generalizes Calabi's
geometry on the space of \K metrics.

\subsection{Geodesics}
\label{gNGeodEqSubsection}

From (\ref{gNLCConnectionEq}) we obtain the geodesic equation
for $(\calM,\gN)$,

\begin{equation}\label{gNGeodEq}
(\ginv g_t)_t=
\frac14\tr(\ginv g_t\ginv g_t)\delta
-\frac12\tr(\ginv g_t)\ginv g_t
+\frac12\NE{g_t}g\ginv g_t
-\frac14\NEsq{g_t}\delta,
\end{equation}
where $\delta$ denotes the Kronecker tensor corresponding
to the identity matrix.
The last two terms are the new terms compared to
the geoedesic equation for $\gE$. Since they are non-local,
the solution of the equation becomes substantially more involved
and requires making use the `constant of motion' of the geodesic
flow found in (\ref{ConstantMotionEq}).

The solution of the initial value problem for the geodesic equation 
is given by the following theorem.

\begin{theo}\label{gNGeod}
  Let $g(0) \in \mathcal{M}$, and let $\mu_0 := dV_{g(0)}$.  Then the
  geodesic in $(\mathcal{M}, \gN)$ emanating from $g(0)$, with initial
  tangent vector $(\alpha, A) \in T_{\mu_0} \Vol(M) \times T_{g(0)}
  \mathcal{M}_{\mu_0}$, is given by
  the following.

  Define $\sigma := \normN{(\alpha, A)}$ and
  \begin{equation*}
    \begin{aligned}
      a_0 &:= \frac{2}{V(0)} \int_M \alpha, & b_0 &:= \sqrt{\frac{n\sigma^2 - a_0^2}{4}},
      & q &:= \frac{\alpha}{\mu_0} - \frac{a_0}{2}, & r &:=
      \sqrt{\frac{n}{4} \tr\big((g(0)^{-1} A)^2\big)}. 
    \end{aligned}
  \end{equation*}

  First, if $b_0 = 0$, then $g(t, x) = e^{t \sigma \sqrt{n}} g(0,x)$.

  If $b_0 \neq 0$, then for each $x \in M$,
      \begin{multline}\label{AneqZeroGeod}
    g(t, x) =
    \left(
      \frac{1}{2}
      \left(
        1 - \frac{q^2 + r^2}{b_0^2}
      \right)
      \cos(b_0 t)
      + \frac{q}{b_0} \sin(b_0 t)
      + \frac{1}{2}
      \left(
        1 + \frac{q^2 + r^2}{b_0^2}
      \right)
    \right)^{\frac{2}{n}} \\
    \cdot e^{a_0 t / n} g(0) \exp
    \left[
      \frac{2}{r}
      \tan^{-1}
      \left(
        \frac{r \sin(b_0 t)}{b_0 + b_0
          \cos(b_0 t) + q \sin(b_0 t)}
      \right)
      g(0)^{-1} A
    \right].
  \end{multline}
  Here, we take the exponential term to be the identity if $A(x) = 0$.
                                
  If $A(x) \neq 0$, then in \eqref{AneqZeroGeod}, arctangent takes
  values in $\left[ \pi k - \frac{\pi}{2}, \pi k + \frac{\pi}{2}
  \right]$ if $t \in \left[ \frac{2 \pi (k-1) + \theta}{b_0}, \frac{2
      \pi k + \theta}{b_0} \right]$ for $k \in \ZZ$, where
  \begin{equation*}
    \theta(x) :=
    \begin{cases}
      2 \pi - \cos^{-1} \left( \frac{q(x)^2 - b_0^2}{q(x)^2 +
          b_0^2} \right) & \textnormal{if } q(x) \geq 0, \\
      \cos^{-1} \left( \frac{q(x)^2 - b_0^2}{q(x)^2 + b_0^2} \right) &
      \textnormal{if } q(x) < 0.
    \end{cases}
  \end{equation*}
  (Here, arccosine takes values in $[0, \pi]$.)

  The domain of definition of $g(t)$ is $[0, \infty)$ if $b_0 = 0$.
  If $b_0 \neq 0$, the domain of definition is $[0, t_0)$, where $t_0$
  is the infimum of $\theta(x)$ at points where $A(x) = 0$.  (We take
  the infimum to be $\infty$ if there are no such points.)  In the
  case where the geodesic exists for only finite time, it approaches a
  limit point on the boundary of $\mathcal{M} \subset \Gamma(S^2 T^*
  M)$ as $t \rightarrow t_0$; i.e., $\mu_{g(t)}(x) \rightarrow 0$ for at 
  least one point $x \in M$.
                    \end{theo}

\begin{remark}
{\rm

Theorem \ref{gNGeod} 
gives a precise meaning to the statement that $\gN$ generalizes Calabi's
geometry on the space of \K metrics. Indeed, on the level of volume
forms, Calabi's geodesics in the space of \K metrics (or, via the
Calabi--Yau Theorem, on the space of volume forms with total volume $v$)
are given by 
$$
dV_{g(t)}
= 
dV_{g}
\Big( 
G\sqrt v \sin\big(\half t/\sqrt v\big)+\cos\big(\half t/\sqrt v\big)
\Big)^2,
$$
where $(dV_{g(0)})_t=GdV_{g(0)}$ \cite[Remark 5.7]{CR}.
On the other hand, in proving (\ref{AneqZeroGeod}) one shows that the volume forms along
$\gN$-geodesics satisfy an equation of a similar 
form---see (\ref{VolFormSoln})---and the two equations can actually
be shown to exactly coincide  
when $A\equiv0$ and $a_0=0$ 
by using trigonometric formulas and carefully identifying the integration constants.
}
\end{remark}

Before we give the proof of this theorem, let us point out a contrast to
the Ebin metric.  Like the case of $\gE$ (cf.~\cite[\S 2]{FG}),
geodesics in $(\mathcal{M}, \gN)$ exist for all time if $A(x) \neq 0$
for all $x \in M$.  However, the converse of this statement also
holds---if $A(x) = 0$ for some $x \in M$, then the geodesic only
exists for finite time---unless $b_0 = 0$.  (In the case of $\gE$,
this happens only when there is a point where $A(x) =0$ and $(\alpha /
\mu_0)(x) < 0$.)

Note also that, as in the case of the Ebin metric, any conformal
class---a submanifold of the form $\mathcal{P} g$ with $g \in
\mathcal{M}$---is totally geodesic, as can be seen
from Theorem \ref{gNGeod} by putting $A\equiv0$.

We will solve the geodesic equation in the following subsections,
beginning with general considerations and then considering various
special cases.

\subsubsection{The general case}
\label{sec:general-case}

We let $C:=\ginv g_t$ and decompose into pure trace and traceless
parts: $C=:E+\frac fnI$, with $\tr E=0$, i.e., $f=\tr C$. From
\eqref{gNGeodEq}, we obtain the pair of coupled equations
\begin{equation}
\label{FirstGeodEquationNE}
E_t=-\frac12fE+\frac E{2V}\int_M fdV_g,
\end{equation}
and
\begin{equation}
\label{SecondGeodEquationNE}
f_t=\frac {n}4\tr(E^2)-\frac{f^2}4-\frac {n\sigma^2}4+\frac f{2V}\int_M fdV_g,
\end{equation}
where $\sigma = \normN{g_t}$ , which is constant since $g$ is a geodesic.
The last term in the first equation and the last two terms in the
second equation are new compared to the unnormalized metric.

The following relations hold between $E$, $f$, and data related to the
splitting $\mathcal{M} \cong \calV \times \mathcal{M}_{\mu_0}$, where
$\mu_0 := dV_{g_0}$ (cf.~\S \ref{PrelimSec}).  We write
$g=\Big(\frac{\mu_g}{\mu_0}\Big)^{2/n}h$, where $\mu_g = dV_g$ and
$h\in\calM_{\mu_0}$, i.e., $h$ is the unique metric conformal to $g$
with $dV_h=\mu_0$.  Then $\ginv
g_t=h^{-1}h_t+\frac2n\frac{(\mu_g)_t}{\mu_g} I$, implying that
$E=h^{-1}h_t$ and $f=2\frac{(\mu_g)_t}{\mu_g}$.

We define
$$
\phi:=f-\frac1V\int_M f \, dV_g.
$$
Note that $V^{-1} \integral{M}{}{f}{dV_g} = 2 \frac{d}{dt} (\log
V_{g(t)})$.  Hence, by Corollary \ref{HessianLogVCor}, this quantity
is constant along $g(t)$.  So defining
\begin{equation*}
  a_0 := \frac{1}{V(0)}\int_M f(0) \, dV_{g(0)},
\end{equation*}
we have $\phi = f - a_0$ and $\phi_t = f_t$.

Now, note that
\begin{equation}
-\frac{f^2}4+\frac{f}{2V}\int_M f
=
-\frac14\phi^2+\frac14\Big(\frac1V\int_Mf\Big)^2.
\end{equation}
Using this, together with the considerations of the previous
paragraph, we can rewrite
(\ref{FirstGeodEquationNE})--(\ref{SecondGeodEquationNE}) in terms of
$\phi$,
\begin{equation}
\label{ModifiedFirstGeodEqNE}
E_t=-\frac\phi2 E,
\end{equation}
\begin{equation}
\label{ModifiedSecondGeodEqNE}
\phi_t
=
\frac n4\tr (E^2)-\frac {n\sigma^2}4-\frac{\phi^2}4+\frac{a_0^2}4. \end{equation}

Note that 
$$
(\tr (E^2))_t=2\tr(E_t E)
=
\left(V^{-1}\int_MfdV_g-f\right)\tr (E^2)=-\phi\,\tr (E^2).
$$
(so
$
\tr (E^2) = \exp\big(\int_0^t(V^{-1}\int_MfdV_g-f)ds\big)\tr (E^2(0))).
$
Hence, differentiating (\ref{ModifiedSecondGeodEqNE}) yields
$$
\phi_{tt}=-\frac n4 \phi\tr (E^2)-\frac12\phi\phi_t,
$$
and substituting for $\tr (E^2)$ using (\ref{ModifiedSecondGeodEqNE}) 
we obtain
\begin{equation}
\label{SecondModifiedSecondGeodEqNE}
4\phi_{tt}+6\phi\phi_t+\phi^3
=
\phi(a_0^2-n\sigma^2).
\end{equation}

We now let
\begin{equation}\label{pVolForm}
  p:=\frac{\mu_g}{\mu_0}\,e^{-a_0t/2}.
\end{equation}
It follows that $2p_t/p=2(\mu_g)_t/\mu_0-a_0 = \phi$.  Thus, the
left-hand side of \eqref{SecondModifiedSecondGeodEqNE} equals
$8p_{ttt}/p$, hence $p$ satisfies
\begin{equation}\label{pttt}
4p_{ttt}-(a_0^2-n\sigma^2)p_t=0.
\end{equation}
Let 
$$
b_0:= \sqrt{\frac{n\sigma^2-a_0^2}4}.
$$
Note that $b_0$ is well-defined, since $\sigma = \normN{g_t(0)}$ and so
\begin{equation}\label{a0LTnSq}
  \begin{aligned}
  \sigma^2 &= \frac{1}{V(0)} \int_M \tr(C(0)^2) \, dV_{g(0)} \geq \frac{1}{n V(0)} \int_M
  f(0)^2 \, dV_{g(0)} \\
  &\geq \frac{1}{n}
  \left(
    \frac{1}{V(0)} \int_M f(0) \, dV_{g(0)}
  \right)^2 = \frac{a_0^2}{n},
\end{aligned}
\end{equation}
where the second inequality is Cauchy--Schwarz.  Note that the first
inequality is an equality if and only if $E(0) \equiv 0$, and the
second inequality is an equality if and only if $f$ is constant.
Therefore, $b_0 = 0$ if and only if $g_t(0) = \lambda g(0)$ for some
$\lambda \in \R$.

Now, integrating \eqref{pttt}, we have
\begin{equation}\label{pODE}
p_{tt}+b_0^2 p=C,
\end{equation}
for some $C\in\RR$. It follows that
\begin{equation}\label{pCases}
p(t)
=
\begin{cases}
C_1\cos(b_0 t)
+C_2\sin(b_0 t)
+C_3, & \h{ if \ } b_0 \neq 0,
\cr
C_1 t^2 + C_2 t + C_3, & \h{ if \ } b_0 = 0.
\end{cases}
\end{equation}
By \eqref{pVolForm}, then,
\begin{equation}\label{VolFormSoln}
\frac{\mu_g}{\mu_0}
=
\begin{cases}
(C_1 \cos(b_0 t)
+ C_2 \sin(b_0 t)
+ C_3) e^{a_0t/2}, & \h{ if \ } b_0 \neq 0,
\cr
(C_1 t^2 + C_2 t + C_3)e^{a_0t/2}, & \h{ if \ } b_0 = 0.
\end{cases}
\end{equation}

We now consider the initial value data needed to determine the
constants of integration.  Note that \eqref{pVolForm}
implies that $p(0) = 1$ and $p_t(0) = \alpha / \mu_0 - a_0 / 2$.

To determine $p_{tt}(0)$, we first use that $\phi = 2 p_t / p$ to see
that on the one hand,
\begin{equation*}
  \phi_t(0) = 2
  \left(
    \frac{p_t}{p}
  \right)_t(0) = 2 \frac{p_{tt}(0) p(0) - (p_t(0))^2}{p(0)^2} = 2
  p_{tt}(0) - 2
  \left(
    \frac{\alpha}{\mu_0} - \frac{a_0}{2}
  \right)^2.
\end{equation*}
On the other hand, we see by \eqref{ModifiedSecondGeodEqNE} and the
fact that $f(0) = 2 \alpha / \mu_0$ that
\begin{equation*}
  \begin{aligned}
    p_{tt}(0) &= \frac{n}{4} \tr(E(0)^2) - \frac{n\sigma^2}{4} - 
    \frac{\phi^2}{4}
    + \frac{a_0^2}{4} + 2
    \left(
      \frac{\alpha}{\mu_0} - \frac{a_0}{2}
    \right)^2 \\
    &= \frac{n}{4} \tr\big((g(0)^{-1} A)^2\big) - \frac{1}{4}
    \left(
      2 \frac{\alpha}{\mu_0} - a_0
    \right)^2 - b_0^2 + 2
    \left(
      \frac{\alpha}{\mu_0} - \frac{a_0}{2}
    \right)^2 \\
    &= \frac{1}{2} (q^2 + r^2 - b_0^2),
  \end{aligned}
\end{equation*}
where $q := \alpha / \mu_0 - a_0 / 2$ and
$r := \sqrt{\frac{n}{4} \tr\big((g(0)^{-1} A)^2\big)}$.

This gives all the information needed to solve for $\mu_g / \mu_0$ in
the individual cases.  To solve for $h$, we must use
\eqref{ModifiedFirstGeodEqNE} and the fact that $\phi = 2 p_t / p$ to
see that $E_t = - (\log p)_t E$, implying $E = E(0) / p = g(0)^{-1} A
/ p$.  Since $E = h^{-1} h_t$, this gives
\begin{equation}\label{hpEqn}
  h^{-1} h_t(t) = g(0)^{-1} A / p(t).
\end{equation}

We now give the solution of the geodesic equation for each special
case.

\subsubsection{The case $b_0 = 0$}
\label{sec:b0Zero}

In this case, we have $a_0 = \sigma \sqrt{n}$ and $g_t(0) = \lambda g(0)$ for
some $\lambda \in \R$, as noted after \eqref{a0LTnSq}.  Therefore, $A
\equiv 0$ and $q \equiv 0 \equiv r$, implying $p_t(0) \equiv 0 \equiv
p_{tt}(0)$.  This gives, in light of \eqref{pCases}, $C_1 = 0 = C_2$,
and $C_3 = 1$.  Thus, $\mu_g / \mu_0 = e^{\sigma \sqrt{n} t / 2}$ by
\eqref{VolFormSoln}, and $h(t) = g(0)$ by \eqref{hpEqn}.  The solution of the geodesic
equation in this case now follows.

\subsubsection{The case $b_0 \neq 0$, $A(x) = 0$}
\label{sec:AZero}

Here, \eqref{pCases} implies that
\begin{equation*}
  \begin{aligned}
    C_1 &= \frac{1}{2}
    \left(
      1 - \frac{q^2}{b_0^2}
    \right), & C_2 &= \frac{q}{b_0}, &
    C_3 &= \frac{1}{2}
    \left(
      1 + \frac{q^2}{b_0^2}
    \right),
  \end{aligned}
\end{equation*}
and thus
\begin{equation*}
  \frac{\mu_g}{\mu_0} =
  \frac{1}{2}
  \left(
    \left(
      1 - \frac{q^2}{b_0^2}
    \right)
    \cos(b_0 t)
    + 2 \frac{q}{b_0} \sin(b_0 t) +
      1 + \frac{q^2}{b_0^2}
  \right)
  e^{a_0 t / 2}.
\end{equation*}

As in the previous case, since $A(x) = 0$, \eqref{hpEqn} gives
$h(t) = g(0)$, so the solution of the geodesic equation in
this case follows.

It remains only to determine the domain of definition of $g(t)$.  The
equation \eqref{AneqZeroGeod} implies $g(t)$ is a smooth Riemannian
metric unless the coefficient of $g(0)$ in that equation vanishes at
some point $x \in M$, which happens if and only if $p(t, x) = 0$.

To see when this occurs in the case we are considering, set $a :=
\cos(b_0 t)$, so that $\sin(b_0 t) = \pm \sqrt{1 - a^2}$.  Setting
$p(t, x)$ equal to zero then leads to the quadratic equation
\begin{equation*}
  \left(
    1 - 2 \frac{q^2}{b_0^2} + \frac{q^4}{b_0^4}
  \right)
  a^2 + 2
  \left(
    1 - \frac{q^4}{b_0^4}
  \right)
  a +
  \left(
    1 + 2 \frac{q^2}{b_0^2} + \frac{q^4}{b_0^4}
  \right)
  = 4 \frac{q^2}{b_0^2} (1 - a^2),
\end{equation*}
or
\begin{equation*}
  \left(
    \left(
      1 + \frac{q^2}{b_0^2}
    \right)
    a +
    \left(
      1 - \frac{q^2}{b_0^2}
    \right)
  \right)^2 = 0.
\end{equation*}
Plugging the solution $\cos(b_0 t) = a = \frac{q^2 - b_0^2}{q^2 +
  b_0^2}$ back into the original equation gives that $\sin(b_0 t)$
must be negative if $q > 0$, and positive if $q < 0$.  Therefore,
letting arccosine take values in $[0, \pi]$, we have that $p(t, x) =
0$ if and only if
\begin{equation}\label{tZero}
  t =
  \begin{cases}
    \frac{1}{b_0}
    \left(
      2 \pi k - \cos^{-1}
      \left(
        \frac{q^2 - b_0^2}{q^2 + b_0^2}
      \right)
    \right),\ k \in \mathbb{Z}, & \textnormal{if}\ q \geq 0, \\
    \frac{1}{b_0}
    \left(
      2 \pi k + \cos^{-1}
      \left(
        \frac{q^2 - b_0^2}{q^2 + b_0^2}
      \right)
    \right),\ k \in \mathbb{Z}, & \textnormal{if}\ q < 0.
  \end{cases}
\end{equation}
We also note that $p(t, x)$ is periodic in $t$, is zero for exactly
one value of $t$ in each period, and is positive for $t = 2 \pi k
/ b_0$.  Therefore, $p(t, x)$ is nonnegative for all $t$.

\subsubsection{The case $A(x) \neq 0$}
\label{sec:ANotZero}

Similarly to the last case, we can compute the constants
$C_1$, $C_2$, and $C_3$ to find
\begin{equation*}
  \frac{\mu_g}{\mu_0} =
  \frac{1}{2}
  \left(
    \left(
      1 - \frac{q^2 + r^2}{b_0^2}
    \right)
    \cos(b_0 t)
    + 2 \frac{q}{b_0} \sin(b_0 t) +
      1 + \frac{q^2 + r^2}{b_0^2}
  \right)
  e^{a_0 t / 2}.
\end{equation*}

Either by integrating \eqref{hpEqn} or directly verifying that the
following solves that equation, one sees that in this case,
\begin{equation}\label{hSoln1}
  \begin{aligned}
  h(t,x) &= g(0, x) \exp
  \left(
    \left( \integral{0}{t}{p(s)^{-1}}{ds} \right) g(0, x)^{-1} A(x)
  \right) \\
  &=
    g(0,x) \exp
  \left[
    \frac{2}{r}
    \left(
      \tan^{-1}
      \left(
        \frac{q}{r} + \frac{q^2+r^2}{b_0 r}
        \tan \left(
          \frac{b_0}{2} t
        \right)
      \right)
    \right.
  \right. \\
  &\qquad \qquad \qquad \qquad \left.
      - \tan^{-1}
      \left(
        \frac{q}{r}
      \right)
    \right)
    g(0)^{-1} A
  \bigg].
      \end{aligned}
\end{equation}
Using the sum formula for arctangent and the half-angle formula for
tangent, we can write this more elegantly as
\begin{equation}\label{hSoln}
  h(t,x) = g(0,x) \exp
  \left[
    \frac{2}{r}
    \tan^{-1}
    \left(
      \frac{r \sin(b_0 t)}{b_0 + b_0
        \cos(b_0 t) + q \sin(b_0 t)}
    \right)
    g(0, x)^{-1} A(x)
  \right].
\end{equation}

As in the last case, \eqref{AneqZeroGeod} implies that $g(t)$ is a
smooth Riemannian metric unless the coefficient of $g(0)$ is
nonpositive.  We claim that in this case, $p(t, x) > 0$ for all $t$, 
implying the coefficient is always positive at $x$.  To see this, we
write 
\begin{equation*}
  p(t, x) = \frac{r^2}{2 b_0} (1 - \cos(b_0 t)) + \frac{1}{2}
  \left(
    \left(
      1 - \frac{q^2}{b_0^2}
    \right)
    \cos(b_0 t)
    + 2 \frac{q}{b_0} \sin(b_0 t) +
    1 + \frac{q^2}{b_0^2}
  \right).
\end{equation*}
Since $r > 0$ in this case, the first term (involving $r$) is always
nonnegative, and it is zero exactly when $t$ is an integer multiple of
$2 \pi / b_0$.

On the other hand, the second term (involving $q$) is formally exactly
the same as $p(t, x)$ from the previous case.  In particular, it is
always nonnegative, and is zero exactly for those values of $t$ given
in \eqref{tZero}.  But this shows that when the first term is zero,
the second term is positive, and vice versa.  Therefore $p(t, x) > 0$
for all $t$.

Finally, to be precise, we must specify the branch of arctangent for
various ranges of $t$ in \eqref{hSoln}.  That entails determining when
the argument of arctangent in \eqref{hSoln} becomes unbounded, and so
we begin by finding when the denominator is zero.  Again substituting
$a := \cos(b_0 t)$ and setting the denominator equal to zero leads to
the quadratic equation
\begin{equation*}
  b_0^2 (1 + a)^2 = q^2 (1 - a^2),
\end{equation*}
which has solutions $a_1 = \frac{q^2 - b_0^2}{q^2 + b_0^2}$ and $a_2 =
-1$.  These two solutions coincide if $q = 0$, and if $q \neq 0$, then
the argument of arctangent in \eqref{hSoln} approaches $r / q$ as $t
\rightarrow \pi = \cos^{-1}(-1)$.  Therefore the argument remains
bounded in this case, and so we are only interested in $a_1$.
Substituting $\cos(b_0 t) = a_1$ and $\sin(b_0 t) = \pm \sqrt{1 -
  a_1^2}$ into $b_0 + b_0 \cos(b_0 t) + q \sin(b_0 t) = 0$ shows that
in this case, $\sin(b_0 t)$ must be negative if $q > 0$ and positive
if $q < 0$.  Note also that as $b_0 t$ approaches $a_1$ from below,
the argument of arctangent in \eqref{hSoln} approaches $+\infty$.
Thus, the branch of arctangent jumps as $t$ approaches the values
\begin{equation}\label{tInfty}
  t =
  \begin{cases}
    \frac{1}{b_0}
    \left(
      2 \pi k - \cos^{-1}
      \left(
        \frac{q^2 - b_0^2}{q^2 + b_0^2}
      \right)
    \right),\ k \in \mathbb{Z}, & \textnormal{if}\ q \geq 0, \\
    \frac{1}{b_0}
    \left(
      2 \pi k + \cos^{-1}
      \left(
        \frac{q^2 - b_0^2}{q^2 + b_0^2}
      \right)
    \right),\ k \in \mathbb{Z}, & \textnormal{if}\ q < 0.
  \end{cases}
\end{equation}
Since $p(t, x) > 0$ for all $t$, the integral $\integral{0}{t}{p(s,
  x)^{-1}}{ds}$ is strictly increasing; therefore, the branch of
arctangent in \eqref{hSoln} ``jumps upwards'' at each value of $t$ in
\eqref{tInfty}.

This completes the analysis of the final case in Theorem
\ref{gNGeod}.

\subsection{Curvature and relation with Calabi's space of \K metrics}

We next restate Proposition \ref{CurvgpProp} in the case $p=1$.

\begin{theo}\label{NormCurv}
  The curvature tensor of $\gN=\gE/V$ is given by
  \begin{equation}
\label{CurvNEq}
    R^{\gN} = \frac{1}{V} R^{\gE} +\frac1{16}\gNE \KN 
     \left( g^\flat \otimes g^\flat - \frac{n}{2} \gN \right),
  \end{equation}
where $g^\flat$ is the 1-form dual to the tautological vector field $g$
with respect to $\gN$.
Let $h$ and $k$ be unit tangent vectors with $\N hk=0$ and $\NEsq h=\NEsq k=1$.
The sectional curvature of the plane $\RR\{h,k\}$ is
\begin{equation}
\label{SecCurvNEq}
\h{\rm sec}^{\gN}(h,k)
=
\frac1V\h{\rm sec}^{\gE}(h,k)
-\frac{{\N gk}^2}{16}-\frac{{\N gh}^2}{16}+\frac n{16}.
\end{equation}
\end{theo}

A conformal class $\calP g$ is totally geodesic (put $A\equiv0$ in (\ref{AneqZeroGeod})).
However, unlike the Ebin metric, it is no longer
flat, and its curvature now changes sign.  Furthermore, since
$\operatorname{sec}^{\gE}$ is nonpositive \cite[Corollary 1.17]{FG}, the
sectional curvature of $\gN$ is bounded above by $\frac{n}{16}$.

Let $(M,J,\o)$ be a compact closed \K manifold of complex dimension
$m=n/2$. Denote by $\calH$ the space of \K metrics cohomologous to
$\o$.  The higher-dimensional Riemannian analogue of $\calH$ is the
space of metrics of fixed volume $v$ within a conformal class, $\calP
g\cap \calM_v$ (where $\calM_v:=\{g\,:\, V_g=v\}$); in fact, these notions
coincide for Riemann surfaces, while in higher dimensions, using the
Calabi--Yau Theorem \cite{Y}, $\calH$
is isometric to $\calP g\cap \calM_v$ 
\cite[\S4.2]{CR}. Now, $\calH$ is not totally geodesic in $(\calM,\gE)$
\cite[\S3]{CR}.
Yet it has
constant positive curvature in the induced metric. This geometry on
$\calH$ is called Calabi's geometry \cite{C1,C2} (see also
\cite{Ca,CR}).  The following corollary describes another sense
(in addition to Theorem \ref{gNGeod}) in
which $\gN$ generalizes Calabi's geometry on the space of \K metrics.
It is one of our motivations in introducing the metric $\gN$.

\begin{cor}
\label{HTotallyGeodCor}
The space of metrics of fixed volume within a conformal class $\calP
g\cap \calM_v$ is totally geodesic in $(\calM,\gN)$, and has constant
curvature $\frac n{16}$.  In particular, when $M$ is a Riemann surface
the space of \K metrics $\calH$ is a totally geodesic portion of a
sphere in $(\calM,\gN)$.
\end{cor}

In fact, for $p=1$ and $m=1$, $\calH\subset \calM$ is the intersection
of the totally geodesic submanifolds $\calM_v$ (cf.~Corollary
\ref{MvTotGeod}) and $\calP g$.

In other words, $\gN$ equips $\calM$ with a geometry for which
the ``Riemannian K\"ahler spaces" $\calP g\cap \calM_v$ (which are
isometric to $\calH$) are totally geodesic
portions of spheres, and in this sense extends Calabi's geometry
to the whole of $\calM$.

\begin{remark}\label{HCurvTotGeod}
{\rm
By (\ref{SecCurvNpEq}), the space $\calP g\cap \calM_v$ has constant curvature
$\frac{np(2-p)}{16V^{1-p}}$
in $(\calM,g_p)$.
However, by adapting the proof of \cite[Proposition 3.1]{CR}, one may readily 
show that this space is no longer totally geodesic for $p\ne 1$.
In a related vein, but with a little more work, one may also show that $\calH$ is no longer totally geodesic
for $\gN$ when $m>1$.

}
\end{remark}

\section{The distance functions and the metric completions}
\label{distance-functions}

In this section, we analyze the distance function $d_p$ of $g_p$,
especially in comparison with the much better-studied distance
function $\dE$ of the Ebin metric.  These distance functions are
defined in the usual way as the infimum of lengths of piecewise
differentiable curves between two points.

Our main result gives one further way that the metric $\gN$ is
distinguished among the family considered in this article.  Namely, the
(metric) completion of $(\mathcal{M}, \dN)$ is strictly smaller than
that of any other $d_p$.  In fact, we will see that for each $p$, the
completion of $(\mathcal{M}, d_p)$ is given by a quotient of the space
of symmetric $(0, 2)$-tensors that are measurable (as sections of $S^2
T^* M$) and positive semi-definite.  (The quotient is given by
identifying tensors that agree wherever they are positive definite;
equivalent tensors may disagree over a set where they are not positive
definite.)  However, if $p = 1$, then the completion consists only of
such tensors with finite, positive total volume.  If $p < 1$, the
completion contains a point representing all such tensors with zero
volume, and if $p > 1$, the completion contains a ``point at
infinity''.  (For precise statements, we refer to \S
\ref{sec:Completion}.)

In the process of proving the completion result, we will show that
$\dE$ and $d_p$, for $p \neq 1$, are equivalent on subsets of metrics
with fixed bounds on their total volume (\S \ref{sec:Equivalence}).
It turns out that $\dE$ and $\dN$ are also equivalent on such subsets, but only
locally (i.e., on small metric balls).  While we suspect $\dE$ and
$\dN$ are inequivalent when considered on the entirety of such a
subset, we have no proof of this fact as yet.

\subsection{The metric completion}
\label{sec:StatementCplResult}

To state the result about the completions of $(\mathcal{M}, d_p)$ in
each of the cases mentioned above, we must introduce some notation.

\begin{definition}\label{MfDef}
  We denote by $\mathcal{M}_f$ the set of measurable,
  positive-semidefinite sections $g : M \rightarrow S^2 T^* M$ with
  finite total volume.  That is, a section $g \in \mathcal{M}_f$ if
  and only if its restriction to any coordinate charts is a measurable
  mapping between subsets of Euclidean space, $g(x)(X, Y) \geq 0$ for
  any $x \in M$ and any $X, Y \in T_x M$, and $V_g = \int_M dV_g <
  \infty$.  Here, $dV_g$ is as usual given locally by $\sqrt{\det g}
  \, dx^1 \wedge \cdots \wedge dx^n$ (which induces a nonnegative
  measure since $g$ is measurable and positive semidefinite).

  We also define $\widehat{\mathcal{M}_f} := \mathcal{M}_f / {\sim}$.
  The equivalence relation ${\sim}$ is defined by $g \sim h$ if and
  only if the following statement holds almost surely (up to a
  Lebesgue-nullset): $g(x) \neq h(x)$ if and only if $\det g(x) = \det
  h(x) = 0$.
\end{definition}

\begin{remark}
  {\rm We note that the concept of a Lebesgue-nullset on a manifold,
    used in the above definition, is well-defined independently of a
    volume form as a set whose image under any coordinate chart is a
    Lebesgue-nullset in $\RR^n$.  }
\end{remark}

We can now state the result, which we will prove in the remainder of
this section.

\begin{theo}\label{dpCpl}
  The metric completion $\overline{(\mathcal{M}, d_p)}$ of
  $(\mathcal{M}, d_p)$ can be identified with
  \begin{enumerate}
  \item $\widehat{\mathcal{M}_{f+}} := \mathcal{M}_{f+} / {\sim}$ if
    $p = 1$, where $\mathcal{M}_{f+} \subset \mathcal{M}_f$ consists
    of those elements with positive total volume;
  \item $\Mfhat$ if $p < 1$;
  \item $\widehat{\mathcal{M}_{f+}} \cup \{ g_\infty \}$ if $p > 1$,
    where $g_\infty$ is a ``point at infinity'' represented by the
    single equivalence class of Cauchy sequences $\{ h_k \}$ with
    $\lim_{k \rightarrow \infty} V_{h_k} = \infty$.
  \end{enumerate}
  In particular, $\overline{(\calM, \gN)}$ is strictly contained in
  $\overline{(\calM, g_p)}$ for all $p \neq 1$.
\end{theo}

For $p \neq 1$, one can very heuristically view these completions as
cones, where for $p < 1$ (resp.~$p > 1$), metrics with zero
(resp.~infinite) volume are identified to a point.  (Of course, there
are other identifications occurring, so this picture is not very
rigorous.)  In the special, scale-invariant case $p = 1$, this cone is
opened to a cylinder.

We begin proving the above theorem by showing the equivalence result
mentioned at the beginning of the section.

\subsection{The (local) equivalence of $d_p$ and $\dE$}
\label{sec:Equivalence}

In this subsection, we show that $d_p$ and $\dE$ are equivalent
metrics---as long as $p \neq 1$---on any subset of $\mathcal{M}$
satisfying an upper and lower bound on the total volume of any element
in the subset.  Furthermore, we will show that for any $p$, they are
equivalent on small balls (of some uniformly positive radius) in such
subsets.  To do so, we first show that the function sending a metric
to its total volume is continuous on $(\mathcal{M}, d_p)$ for any $p$.
This allows us to prove the uniform local equivalence for any $p$
mentioned above.  Following that, we state a result that, in
particular, implies that subsets of metrics with certain bounds on
their total volumes have bounded diameter with respect to both $d_p$
and $\dE$, for $p \neq 1$.  (It is at this point that the proof fails
for $p = 1$; however, we do not yet know whether $p \neq 1$ is an
essential assumption.)  A simple metric space argument then gives the
global equivalence on the subsets we are considering.

We begin this process with the following lemma, which was inspired by
\cite[\S 3.3]{MM2} and generalizes \cite[Lemma 12]{Cl2}.

\begin{lem}\label{VolLipCont}
  Let $g, h \subseteq \mathcal{M}$.  Then
  \begin{equation*}
    d_p(g, h) \geq
    \begin{cases}
      \frac{4}{(1-p)\sqrt{n}} \abs{V_{h}^{\frac{1-p}{2}} - V_{g}^{\frac{1-p}{2}}}, & p \neq 1, \\
      \frac{2}{\sqrt{n}} \abs{\log\left( \frac{V_h}{V_g}
        \right)}, & p = 1.
    \end{cases}
  \end{equation*}
  In particular, the function $g \mapsto V_g$ is continuous on
  $(\mathcal{M}, d_p)$.
\end{lem}
\begin{proof}
  Let $\gamma(t)$, $t \in [0,1]$, be any path from $g$ to $h$, and
  define $k(t) := \gamma_t(t)$.  We
  compute
  \begin{equation}\label{eq:26}
      \partial_t V_{\gamma(t)} = \frac{1}{2} \int_M \tr(\gamma^{-1} k)
      \, dV_\gamma \leq \frac{1}{2} \sqrt{V_\gamma} \left( \int_M
        (\tr(\gamma^{-1} k))^2 \, dV_\gamma \right)^{1/2},
  \end{equation}
  where we have used H\"older's inequality in the second line.

  Let $k_0(t)$ denote the trace-free part of $k(t)$.  By the
  orthogonality of traceless and trace-free matrices in the
  Hilbert--Schmidt product $\langle A, B \rangle = \tr(AB^T)$, and
  since $k = k_0 + \frac{1}{n} \tr(g^{-1} k) \gamma$, we have
  \begin{equation*}
    (\tr(\gamma^{-1} k))^2 = n \left( \tr((\gamma^{-1} k)^2) -
      \tr((\gamma^{-1} k_0)^2) \right) \leq n \tr((\gamma^{-1} k)^2).
  \end{equation*}
  Applying this to (\ref{eq:26}) gives
  \begin{equation*}
      \partial_t V_{\gamma(t)} \leq \frac{1}{2} \sqrt{V_\gamma}
      \left( n \int_M \tr((\gamma^{-1} k)^2) \, dV_\gamma \right)^{1/2}
      \leq \frac{\sqrt{n}}{2} \sqrt{V_\gamma} \abs{k}_E =
      \frac{\sqrt{n}}{2} V_{\gamma(t)}^{\frac{1+p}{2}} \abs{k}_p.
  \end{equation*}

  Now, let $p \neq 1$.  We estimate
  \begin{equation}\label{eq:49}
    \begin{aligned}
      V_{h}^{\frac{1-p}{2}} - V_{g}^{\frac{1-p}{2}} &= \int_0^1 \partial_t
      V_{\gamma(t)}^{\frac{1-p}{2}} \, dt = \frac{1-p}{2}\int_0^1 \partial_t
        V_{\gamma(t)} V_{\gamma(t)}^{\frac{-1-p}{2}} \, dt \\
      &\leq \frac{(1-p)\sqrt{n}}{4} \int_0^1 \abs{k(t)}_p \, dt 
      = \frac{(1-p)\sqrt{n}}{4} L_p(\gamma).
    \end{aligned}
  \end{equation}
  Since this inequality holds for all paths from $g$ to $h$, and we can
  repeat the computation with $g$ and $h$ interchanged, it implies
  the result for $p \neq 1$.   The case $p = 1$ follows analogously to \eqref{eq:49} if one begins
  with the quantity $\log(V_{h}) - \log(V_{g})$ on the left-hand side.
\end{proof}

The following is an immediate corollary, and hints at the completions
described in the introduction to this section.

\begin{cor}\label{CauchyVolConv} 
  If $\{ h_k \} \subset \mathcal{M}$ is a $d_p$-Cauchy sequence, then
  $\{ V_{h_k} \}$ converges in $\R_+ \cup \{ 0 \}$ {\rm(}for $p < 1${\rm)},
  $\R_+$ {\rm(}for $p = 1$\h{\rm)}, or $\R \cup \{ +\infty \}$ {\rm(}for $p > 1${\rm)}.
\end{cor}

Lemma \ref{VolLipCont} also yields the following comparison
between $d_p$ and $\dE$.

\begin{cor}\label{dsdEEquivLocal}
  Let $v' > v > 0$ be given.  Define $\mathcal{M}_{v,v'} := \{ g \in
  \mathcal{M} \,:\, v < V_g < v' \}$ Then there exists $\delta =
  \delta(v, v') > 0$ such that if $g \in
  \mathcal{M}_{v,v'}$ and $h \in \mathcal{M}$, then
  \begin{enumerate}
  \item\label{dsLTdE} $\dE(g, h) < \delta$ implies $d_p(g, h) < \max\Big\{ (2v')^{-p},
    \left(\frac{v}{2}\right)^{-p} \Big\} \dE(g, h)$, and
  \item\label{dELTds} $d_p(g, h) < \delta$ implies $\dE(g, h) < 
    \max\left\{ (2v')^p,
    \left(\frac{v}{2}\right)^p \right\} d_p(g, h)$.
  \end{enumerate}
\end{cor}
\begin{proof}
  By Lemma \ref{VolLipCont},
    the function $g
  \mapsto V_g$ is uniformly continuous with respect to both $d_p$ and
  $\dE$ on $\mathcal{M}_{v,v'}$.  So we can choose $\delta$ small
  enough that if $g \in \mathcal{M}_{v,v'}$, $h \in \mathcal{M}$, and either $d_p(g,
  h) < 2 \delta$ or $\dE(g, h) < 2 \delta$, then $\frac{v}{2} <
  V_{h} < 2v'$.

  Let $g$, $h$, and $\delta$ be as above, let $0 < \epsilon < \delta$
  be arbitrary, and let $\{ \gamma(t) \}_{t \in [0,1]}$ be any
  piecewise differentiable path connecting $g$ and $h$ that satisfies
  $\LE(\gamma) < \dE(g, h) + \epsilon$, where we denote by $\LE$ and
  $L_p$ the length with respect to $\gE$ and $g_p$, respectively.
  Since $\dE(g, \gamma(t)) < 2 \delta$ for any $t \in [0,1]$,
  $\frac{v}{2} < V_{\gamma(t)} < 2v'$ for all $t$.  Thus we may
  estimate
  \begin{equation*}
    L_p(\gamma(t)) = \int_0^1 \abs{\gamma_t(t)}_p \, dt = \int_0^1 V^{-p}
    \abs{\gamma_t(t)}_E \, dt \leq \max\Big\{ (2v')^{-p},
    \left(\frac{v}{2}\right)^{-p} \Big\} \LE(\gamma(t)).
  \end{equation*}
  Since $\LE(\gamma(t)) < \dE(g, h) + \epsilon$ and $\epsilon$
  was arbitrarily small, this proves statement \eqref{dsLTdE}.  Statement
  \eqref{dELTds} is then proved completely analogously.
\end{proof}

Since $g_p$ is, as discussed in \S \ref{PrelimSec}, a weak Riemannian
metric, the distance function $d_p$ does not a priori induce a metric
space structure on $\mathcal{M}$ (as it is not a priori positive
definite; the other metric space axioms are automatic).  In fact,
there are examples (e.g., due to Michor--Mumford \cite{MM1, MM2}) of
weak Riemannian manifolds with induced distance between any two points
zero.  However, it has been shown \cite[Theorem 18]{Cl2} that $\dE$ does
induce a metric space structure on $\mathcal{M}$, and so Corollary
\ref{dsdEEquivLocal} gives:

\begin{cor}\label{dpMetricSpace}
  $(\mathcal{M}, d_p)$ is a metric space.
\end{cor}

We now give a proposition that estimates $d_p$ from above in a way
that is, at least in spirit, converse to Lemma \ref{VolLipCont}.  This
proposition allows us to bound the distance between two metrics based
only on their total volumes and the intrinsic volumes of the set on
which they differ.  A direct consequence is a diameter bound for
subsets of metrics satisfying a bound on their total volumes.

\begin{prop}\label{VolUpperBdpNE1}
  Suppose that $g, h \in \mathcal{M}$, and let $E :=
  \operatorname{carr} (h - g) = \{ x \in M \mid g(x) \neq h(x) \}$.
  If $p \neq 1$, then there exists a constant $C(p, n)$, depending
  only on $p$ and $n = \dim M$, such that
  \begin{equation*}
    d_p(g, h) \leq  C(p, n) \cdot \left( V_g^{-p/2} \sqrt{\Vol(E, g)}
      + V_h^{-p/2} \sqrt{\Vol(E,h)} \right).
  \end{equation*}

  In particular, let $0 < v < \infty$.  Then if $p < 1$, we
  have
  \begin{equation*}
    \diam_{d_p} \left( \{ \tilde{g} \in \mathcal{M} \mid \Vol(M,
      \tilde{g}) \leq v \} \right) \leq 2 C(p, n) v^{\frac{1-p}{2}}.
  \end{equation*}
  If $p > 1$, then we have
  \begin{equation*}
    \diam \left( \{ \tilde{g} \in \mathcal{M} \mid \Vol(M, \tilde{g}) \geq
      v \} \right) \leq 2 C(p, n) v^{\frac{1-p}{2}}.
  \end{equation*}
\end{prop}

Since the proof of this proposition is rather lengthy, we postpone
it to the Appendix.

Corollary \ref{dsdEEquivLocal} and Proposition \ref{VolUpperBdpNE1}
imply, with just a little extra work, that $d_p$ ($p \neq 1$) and
$\dE$ are equivalent on the sets $\mathcal{M}_{v,v'}$ defined in
Corollary \ref{dsdEEquivLocal}.

\begin{cor}\label{dsdEEquiv}
  Let $p \neq 1$ and $0 < v, v' < \infty$.  Then $d_p$ and $\dE$ are
  equivalent on $\mathcal{M}_{v,v'}$ (where by $d_p$ and $\dE$ we mean
  the extrinsic distance induced by $g_p$ and $\gE$, respectively, on
  this subset).
\end{cor}
\begin{proof}
  Let $g, h \in \mathcal{M}_{v,v'}$.
  
  Corollary \ref{dsdEEquivLocal} implies that there exist $\epsilon >
  0$ and $1 \leq \eta < \infty$ such that if either $d_p(g, h) \leq
  \epsilon$ or $\dE(g, h) \leq \epsilon$, then
  \begin{equation}\label{LocalEquivIneq}
    \eta^{-1} d_p(g, h) \leq \dE(g, h) \leq \eta d_p(g, h).
  \end{equation}

  On the other hand, let $\dE(g, h) > \epsilon$; then the preceding
  paragraph gives $d_p(g, h) > \eta^{-1} \epsilon$.  Furthermore,
  Proposition \ref{VolUpperBdpNE1} implies that there exists $D <
  \infty$ such that the diameter of $\mathcal{M}_{v,v'}$ is at most
  $D$ with respect to both $d_p$ and $\dE$, so we also have $\dE(g,
  h), d_p(g, h) \leq D$.  Thus,
  \begin{equation*}
    d_p(g, h) > \eta^{-1} \epsilon = \frac{\eta^{-1} \epsilon}{D} D
    \geq \frac{\eta^{-1} \epsilon}{D} \dE(g, h),
  \end{equation*}
  and
  \begin{equation*}
    d_p(g, h) \leq D = \frac{D}{\epsilon} \epsilon <
    \frac{D}{\epsilon} \dE(g, h).
  \end{equation*}
  This completes the proof.
\end{proof}

\subsection{The completion of $(\calM, d_p)$}
\label{sec:Completion}

Using these results, together with the characterization of the
completion of $(\mathcal{M}, \dE)$ in \cite{Cl3}, we can prove \ref{dpCpl}.

First, though, we need to recall the
completion of $(\mathcal{M}, \dE)$, as determined in \cite{Cl3}.  This
requires some background discussion.

\begin{definition}\label{OmegaDef}
  Let $\mathcal{M}_x := S^2_+ T^*_x M$ denote the set of
  positive-definite $(0, 2)$-tensors at $x \in M$; its tangent spaces
  are given by $T_a \mathcal{M}_x \cong S^2 T^*_x M$.  Define a
  Riemannian metric $\langle \cdot, \cdot \rangle$ on $\mathcal{M}_x$
  by $\langle b, c \rangle_a := \tr(a^{-1} b a^{-1} c) \sqrt{\det
    (\tilde{g}(x)^{-1} a)}$, where $\tilde{g} \in \mathcal{M}$ is any fixed
  reference metric.

  Let $d_x$ denote the distance function of $\langle \cdot, \cdot
  \rangle$ on $\mathcal{M}_x$.  Define a metric (in the sense of
  metric spaces) on $\mathcal{M}$ by
  \begin{equation*}
    \Omega_2(g, h) :=
    \left(
      \integral{M}{}{d_x(g(x), h(x))^2}{dV_{\tilde{g}}}
    \right)^{1/2}.
  \end{equation*}
\end{definition}

It is not hard to see that $\Omega_2$ is indeed a metric, and one can
show that it does not depend on the arbitrary choice of $\tilde{g}$
(see~\cite{Cl5}).  The completion of $(\mathcal{M}_x, d_x)$ is given
by $\operatorname{cl}(\mathcal{M}_x) / \partial \mathcal{M}_x$, that
is, by all positive-semidefinite $(0, 2)$-tensors at $x$, with tensors
that are not positive definite identified to a point.  A sequence $\{
a_k \} \subset \mathcal{M}_x$ converges in the completion to $[0]$,
the equivalence class of the zero tensor, if and only if
$\det(\tilde{g}(x)^{-1} a_k) \rightarrow 0$ \cite[Proposition
18]{Cl5}.  One can use this fact to show that the metric $\Omega_2$
can also be extended in a well-defined way to $\Mfhat$ \cite[\S
4.1]{Cl5}.

In fact, we have the following theorem, which in particular
says that, like curvature and geodesics, the distance between points
(and in a sense the completion) of $(\mathcal{M}, \gE)$ can be
computed ``fiberwise''.

\begin{theo}[{\cite[Theorem~5.17]{Cl3}, \cite[Theorem~22]{Cl5}}]\label{gECompletion}
  For all $g, h \in \mathcal{M}$, $\dE(g, h) = \Omega_2(g, h)$.

  The metric completion $\overline{(\mathcal{M}, \gE)}$ of
  $(\mathcal{M}, \gE)$ is identified with $\Mfhat$.  That is, for each
  $\dE$-Cauchy sequence $\{ h_k \} \subset \mathcal{M}$, there exists
  a unique element $h \in \Mfhat$ such that $\Omega_2(h_k, h)
  \rightarrow 0$.  Furthermore, if $\{ \tilde{h}_k \} \subset
  \mathcal{M}$ is another $\dE$-Cauchy sequence with $\lim_{k
    \rightarrow \infty} d(h_k, \tilde{h}_k) = 0$, then
  $\Omega_2(\tilde{h}_k, h) \rightarrow 0$ as well.
\end{theo}

Using the (local) equivalence of $\dE$ and $d_p$, as well as the
completion of $(\mathcal{M}, \gE)$ as a basis for comparison, we can
now prove Theorem \ref{dpCpl}.

\begin{proof}[Proof of Theorem \ref{dpCpl}]
  We begin with general arguments.  Following that, we treat the
  specifics of each of the three cases.
  
  Let $\{ h_k \}$ be a $d_p$-Cauchy sequence.  By Corollary
  \ref{CauchyVolConv}, $\{ V_{h_k} \}$ converges either to a
  nonnegative real number or infinity.  Let's assume that it converges
  to a positive number.  Then there exist $0 < v \leq v' < \infty$
  such that $\{ h_k \} \subset \mathcal{M}_{v,v'}$ (with notation as in
  Corollary \ref{dsdEEquivLocal}).  But then Corollary
  \ref{dsdEEquivLocal} 
implies that $\{ h_k \}$ is
  $\dE$-Cauchy as well.  Therefore, by Theorem \ref{gECompletion}, $\{
  h_k \}$ $\Omega_2$-converges to a unique limit point $h$ in
  $\widehat{\mathcal{M}_f}$ with $V_h > 0$.  This shows there exists a
  mapping from the set of $d_p$-Cauchy sequences in $\mathcal{M}$ with
  positive volume in the limit to $\widehat{\mathcal{M}_{f+}}$.

  To see that this induces a well-defined mapping from a subset of the
  completion $\overline{(\mathcal{M}, d_p)}$ to $\Mfplushat$, we must
  show that if $\{ h_k \}$ and $\{ \tilde{h}_k \}$ are $d_p$-Cauchy
  sequences with positive volume in the limit and $\lim_{k \rightarrow
    \infty} d_p(h_k, \tilde{h}_k) = 0$, then $\{ h_k \}$ and $\{
  \tilde{h}_k \}$ $\Omega_2$-converge to the same element $h \in
  \Mfplushat$.  But 
 in this case there exist $0 < \tilde{v}
  \leq \tilde{v}' < \infty$ such that $\{ h_k \}$ and $\{ \tilde{h}_k
  \}$ both lie in $\mathcal{M}_{\tilde{v},\tilde{v}'}$, so this is
  implied by Corollary \ref{dsdEEquivLocal} and Theorem
  \ref{gECompletion}.

  On the other hand, the same argument, with the roles of $\dE$ and
  $d_p$ reversed, shows that if $\{ h_k \}$ is a $\dE$-Cauchy
  sequence with $\lim_{k \rightarrow \infty} V_{h_k} > 0$, then $\{
  h_k \}$ is $d_p$-Cauchy.  Therefore, the mapping from this subset of
  $\overline{(\mathcal{M}, d_p)}$ to $\Mfplushat$ is
  surjective.

  To see that the mapping from this subset of $\overline{(\mathcal{M},
    d_p)}$ to $\Mfplushat$ is injective, we must show that if $\{ h_k
  \}$ and $\{ \tilde{h}_k \}$ are Cauchy sequences with positive
  volume in the limit and $\lim_{k \rightarrow \infty} d_p(h_k,
  \tilde{h}_k) \neq 0$, then the $\Omega_2$-limits of $\{ h_k \}$ and
  $\{ \tilde{h}_k \}$ differ.  But as in the proof that the mapping is
  well-defined, this follows from Corollary \ref{dsdEEquivLocal} and
  Theorem \ref{gECompletion}.

  Now, consider the case $p = 1$.  Here, Corollary \ref{CauchyVolConv}
  implies that all Cauchy sequences have positive volume in the limit,
  so the preceding arguments suffice for this case.

  If $p < 1$, the only remaining $d_p$-Cauchy sequences $\{ h_k \}$
  are those for which $\lim_{k \rightarrow \infty} V_{h_k} = 0$, again by
  Corollary \ref{CauchyVolConv}.  To complete the proof of the
  theorem, we must show that if $\{ h_k \}$ and $ \{ \tilde{h}_k \}$
  are two such sequences, then $\lim_{k \rightarrow \infty} d_p(h_k,
  \tilde{h}_k) = 0$.  But this follows from Proposition
  \ref{VolUpperBdpNE1}.

The case $p>1$ follows 
from the case $p<1$ using the
  isometry of Proposition \ref{IsometryProp}.
\end{proof}

\section{Remarks and open questions}
\label{sec:Remarks}

\subsection{(Non-)Control over geometry via $d_p$}
\label{sec:GeomControl}

In \cite[Example 4.17]{Cl4}, it was shown that the metric $\dE$ is too
weak to control, in any reasonable way, various geometric quantities
associated to elements of $\mathcal{M}$.  That is, functions mapping
a metric in $\mathcal{M}$ to its curvature, distance function,
diameter, or injectivity radius are discontinuous, even in some
weakened sense.

In fact, the same examples constructed in op.~cit.~for
$\dE$ are also valid for $d_p$.  To see this, and make it precise, we
give a result analogous to Proposition \ref{VolUpperBdpNE1}, with a
statement weakened in order to handle the case $p = 1$.  It only gives
an upper bound on the distance between metrics that agree as tensors
somewhere on $M$.  On the other hand, if two metrics differ everywhere
(the generic case), this proposition gives no information.

\begin{prop}\label{VolUpperBd}
  Suppose that $g, h \in \mathcal{M}$, and let $E :=
  \operatorname{carr} (h - g) = \{ x \in M \mid g(x) \neq h(x)
  \}$.  Given a measurable subset $A \subseteq M$ and $\tilde{g} \in
  \mathcal{M}$, let 
  \begin{equation*}
    V_{p,\tilde{g}}^A := \max \left\{ V_{\tilde{g}}^{-p/2},
      \Vol(M \setminus A, \tilde{g})^{-p/2} \right\}.
  \end{equation*}
  Then there exists a constant $C(n)$, depending only on $n = \dim M$,
  such that
  \begin{equation*}
    d_p(g, h) \leq  C(n) \cdot \left( V_{p,g}^E \sqrt{\Vol(E, g)}
      + V_{p,h}^E \sqrt{\Vol(E,h)} \right).
  \end{equation*}
          \end{prop}

The proof of this proposition is postponed to the Appendix.

In op.~cit.,
taking $M = T^2$, the two-dimensional torus,
several examples of sequences $\{ h_k \} \subset \mathcal{M}$ with the
following properties were constructed:
\begin{itemize}
\item $dV_{h_k} = dV_h$ for all $k \in \NN$, where $h$ denotes the
  standard flat metric on $T^2$ (with both radii equal to 1);
\item for each $k \in \NN$, there exists a set $U_k \subseteq M$ with
  $h_k = h$ off of $U_k$; and
\item $\Vol(U_k, h) \rightarrow 0$ as $k \rightarrow \infty$.
\end{itemize}
The above properties imply, by Proposition \ref{VolUpperBd}, that
$d_p(h_k, h) \rightarrow 0$.  Furthermore, various sequences with the
above properties were constructed so that, depending on the sequence,
\begin{itemize}
\item no curvature quantity of $(M, h_k)$ converges to the
  corresponding quantity for $(M, h)$, even off of some small-measure subset;
\item the distance function induced on $M$ by $h_k$ does not converge
  to that of $h$, either in the Gromov--Hausdorff sense or some sense
  relevant to metric-measure spaces;
\item $\diam(M, h_k)$ does not converge to $\diam(M, h)$; or
\item the injectivity radius of $(M, h_k)$ does not converge to that
  of $(M, h)$, either as a function of $M$ off of some small-measure
  subset, or taking the infimum of this function.
\end{itemize}

Since these examples apply to $d_p$, it seems the advantage of
$d_p$, when considered in the context of convergence of Riemannian
manifolds, is that it eliminates collapse of the metrics over the
entire manifold if $p = 1$.

To the best of our knowledge, it remains an open question to find a
simple Riemannian metric on $\mathcal{M}$ with a distance function
that offers some control over the geometry of elements of
$\mathcal{M}$---for instance, one for which convergence with respect
to the distance function of the Riemannian metric implies
Gromov--Hausdorff convergence (or some other synthetic-geometric
convergence).  While this is certainly the case for Sobolev $H^s$
metrics when $s > n/2$ (cf.~\cite[p.~20]{E} or \cite{BHM}), it might
be the case that there are simpler Riemannian metrics with this desirable
property.  (Compare \cite{MM2,MM1} for analogous examples of this in the
setting of submanifold geometry.)

\subsection{The exponential mapping of $\gN$}
\label{sec:expon-mapp-gn}

It is possible, though a bit tricky (see the next two subsections), to
see that the exponential mapping of $\gN$ is surjective onto any
conformal class, but not onto all of $\mathcal{M}$.  This is also true
for the Ebin metric.  It would be interesting to find a $\operatorname{Diff}(M)$-invariant
geodesically convex Riemannian metric on $\mathcal{M}$, that is, one
for which geodesics exist between any two points.  However, at this
point the authors know of no such metric.

\subsubsection{Conformal classes}
\label{sec:conformal-classes}

Let us now show that for any $g \in \mathcal{M}$, $\exp_g$ is a
diffeomorphism onto the conformal class $\mathcal{P} g$ of $g$, when
restricted to an appropriate open neighborhood of $0$ in $T_g
(\mathcal{P}g)$.  (The same is true for the Ebin metric, as is
immediately apparent from the explicit formula for its exponential
mapping \cite[Theorem 2.3]{FG}, \cite[Theorem 3.2]{GM}.)  We show this
in the remainder of this subsection.

Indeed,  the completion of the
set $\mathcal{V}_v$ of smooth volume forms with fixed total volume $v
= \Vol(M, g)$ is isometric to a section of a sphere in a Hilbert space
when endowed with the metric induced from the Ebin metric via the map
$i_\mu$ (\ref{imuEq}) \cite[\S 4.4]{CR}.  
In particular, one can deduce
that the exponential mapping of $\mathcal{V}_v$ is a
diffeomorphism from a subset of $T_\nu \mathcal{V}_v$ onto
$\mathcal{V}_v$ for any $\nu \in \mathcal{V}_v$.

Consider now the set $\calP g \cap \calM_v = \{ h \in \mathcal{P} g \,
: \, \Vol(M, h) = v \}$.  Since the metric induced by $\gN$ on $\mathcal{P}
g \cap \mathcal{M}_v$ is equal (up to a factor $1/v$) to the Ebin metric, and $i_\mu$ induces
a diffeomorphism between $\calP g \cap \calM_v$ and $\mathcal{V}_v$,
one also sees that the exponential mapping at $g$ of $(\calP g \cap
\calM_v, \gN)$ is a diffeomorphism when restricted to the appropriate
domain.  Furthermore, as noted above Remark \ref{HCurvTotGeod},
$\mathcal{P} g \cap \mathcal{M}_v \subset \mathcal{M}$ and
$\mathcal{P}g \subset \mathcal{M}$ are totally geodesic.  Therefore,
the exponential mapping of $(\mathcal{M}, \gN)$, restricted to vectors
tangent to $\mathcal{P} g \cap \mathcal{M}_v$, coincides with that of $(\calP g \cap
\calM_v, \gN)$.

Now, let notation be as in Theorem \ref{gNGeod}, and let $\{ g(t)
\}_{t \in [0,1]}$ be any geodesic emanating from $g(0) = g$ with
initial tangent vector $(\alpha, 0)$, where $a_0 = \frac{2}{v}
\integral{M}{}{}{\alpha} = 0$; that is, $(\alpha, 0)$ is tangent to
$\calP g \cap \calM_v$ and $\Vol(M, g(t)) = v$ for all $t$.  Let us now
consider the geodesic $\tilde{g}(t)$ emanating from $g(0)$ with
initial tangent vector $(\alpha + \lambda \mu_0, 0)$, where $\lambda
\in \RR$.  One then computes that under this change, $a_0$
becomes $2 \lambda$, but $b_0$, $q$, and $r$ do not change.  Examining
\eqref{AneqZeroGeod}, then, $\tilde{g}(t) =
e^{\lambda t / n} g(t)$.  Since $\mathcal{P} g = \RR_{>0} \cdot
(\calP g \cap \calM_v)$, one deduces that $\exp_g$ is a
diffeomorphism from an appropriate domain in $T_g (\mathcal{P} g)$
onto $\mathcal{P} g$.

\subsubsection{Nonsurjectivity on $\calM$}
\label{sec:nonsurjectivity-calm}

To show that for no $g \in \mathcal{M}$ is $\exp_g$ surjective onto
$\mathcal{M}$, we continue to use the notation of Theorem
\ref{gNGeod}, and consider any geodesic $\{ g(t) \}_{t \in [0, T)}$ with
$g(0) = g$ and $g_t(0) = (\alpha, A)$.  Let $\lVert A(x) \rVert :=
\sqrt{\tr((g(0, x)^{-1} A(x))^2)}$ denote the fiberwise norm of $A$,
and $\overline{A}(x) := A(x) / \lVert A(x) \rVert$ the fiberwise
normalization of
$A$; then $\frac{2}{r} g(0)^{-1} A = \frac{4}{\sqrt{n}} g(0)^{-1}\overline{A}$.

Now, recall that the branch of arctangent in \eqref{AneqZeroGeod}
``jumps upward'' when $t \mapsto t + \frac{2\pi}{b_0}$.  Furthermore,
its argument has period $\frac{2\pi}{b_0}$; therefore the arctangent
term increases by adding $\pi$ when $t \mapsto t + \frac{2 \pi}{b_0}$.
In particular, using the considerations of the previous paragraph as
well, we have $g \left( \frac{2 \pi k}{b_0} \right) = g(0) \exp \left(
  \frac{4 \pi k}{\sqrt{n}} g(0)^{-1} \overline{A} \right) $ for any $k \in \NN$.

To complete the proof of non-surjectivity, note that at each $x \in
M$, $g(t, x) = a(t, x) g(0, x) \exp(b(t, x) g(0, x)^{-1}
\overline{A}(x))$, where $a$ and $b$ are real-valued functions.  
Furthermore, from \eqref{hSoln1} (and the nonnegativity of $p$ in that equation),
it follows that $b(\, \cdot \, , x)$ is monotonically nondecreasing
for each $x \in M$.  From the last paragraph, we also see that $b
\left( \frac{2\pi k}{b_0}, x \right) = \frac{4 \pi k}{\sqrt{n}}$ for
any $x \in M$ and $k \in \NN$.  Since also $\lVert \overline{A}(x)
\rVert = 1$ for all $x \in M$, we see that it is impossible for the
image of $\exp_g$ to contain, for example, any metrics of the form $R
g(0) \exp(S)$, where $R: M \rightarrow \RR_{>0}$ and $S$ is any $(1,
1)$-tensor with $\sqrt{\tr(S^2(x))} < \frac{4 \pi k_0}{\sqrt{n}}$ and
$\sqrt{\tr(S^2(y))} > \frac{4 \pi k_0}{\sqrt{n}}$ for some points $x,
y \in M$ and number $k_0 \in \NN$.

\section*{Appendix}
\label{sec:appendix}

Here, we present the proofs of Propositions \ref{VolUpperBd} and
\ref{VolUpperBdpNE1}.

\begin{proof}[Proof of Proposition \ref{VolUpperBd}]
  This proposition is analogous to \cite[Proposition 4.1]{Cl3}, so we will
  follow that proof, with modifications to compensate for
  the conformal factor $V^{-p}$ of $g_p$.

  For each $k \in \NN$ and $s \in (0,1]$, we define three families of
  metrics as follows.  The set $E$ is open, and we may choose closed sets
  $F_k \subseteq E$ such that $\Vol(E, g) - \Vol(F_k, g) \leq 1/k$.
  (This is possible because the Lebesgue measure is regular.)  Let
  $f_{k,s} \in C^\infty(M)$ be functions with the following
  properties:
  \begin{enumerate}
  \item $f_{k,s}(x) = s$ if $x \in F_k$,
  \item $f_{k,s}(x) = 1$ if $x \not\in E$ and
  \item $s \leq f_{k,s}(x) \leq 1$ for all $x \in M$.
  \end{enumerate}
  Now, for $t \in [0,1]$, define
  \begin{equation*}
    \begin{gathered}
      \hat{g}^{k,s}(t) := ((1-t) + t f_{k,s}) g, \qquad
      \bar{g}^{k,s}(t) := f_{k,s} ((1-t) g + t h), \\
      \tilde{g}^{k,s}(t) := ((1-t) + t f_{k,s}) h.
    \end{gathered}
  \end{equation*}
  We view these as paths in $t$ depending on the family parameter
  $s$.  Furthermore, we define a concatenated path
  \begin{equation*}
    g^{k,s} := \hat{g}^{k,s} * \bar{g}^{k,s} * (\tilde{g}^{k,s})^{-1},
  \end{equation*}
  where of course the inverse means we run through the path backwards.
  Then $g^{k,s}(0) = g$ and $g^{k,s}(1) = h$ for
  all $s$.

  We now investigate the lengths of each piece of $g^{k,s}$
  separately, starting with that of $\hat{g}^{k,s}$.  We first compute
  \begin{equation}\label{HatgEst}
    \begin{aligned}
      L(\hat{g}^{k,s})       &= \integral{0}{1}{\left( V_{\hat{g}^{k,s}(t)}^{-p} \integral{M}{}{\tr_{((1 - t) + t
              f_{k,s}) g} \left( ((f_{k,s} - 1)g)^2 \right)
                      }{dV_{\hat{g}^{k,s}(t)}} \right)^{1/2}}{d t} \\
      &= \integral{0}{1}{\left( V_{\hat{g}^{k,s}(t)}^{-p} \integral{E}{}{((1 - t) + t
            f_{k,s})^{\frac{n}{2} - 2} \tr_{g} \left( ((1 -
              f_{k,s})g)^2 \right)
                      }{dV_{g}}
        \right)^{1/2}}{d t}.
    \end{aligned}
  \end{equation}
  Note that in the last line, we only integrate over $E$, since $1 -
  f_{k,s} \equiv 0$ on $M \setminus E$.  Note also that since,
  additionally, $f_{k,s} \leq 1$, we have $\Vol(M \setminus E, g)
  \leq V_{\hat{g}^{k,s}(t)} \leq V_{g}$.  Furthermore, since $s > 0$,
    we have $(1 - f_{k,s})^2 \leq (1 - s)^2 < 1$, from
  which
  \begin{equation*}
    L(\hat{g}^{k,s}) < V_{p,g}^{E} \integral{0}{1}{\left( n \integral{E}{}{((1
          - t) + t f_{k,s})^{\frac{n}{2} - 2}}{dV_{g}}
      \right)^{1/2}}{d t}.
  \end{equation*}
  Now, to estimate this, we note that for $n \geq 4$, $\frac{n}{2} - 2
  \geq 0$ and therefore $f_{k,s} \leq 1$ implies that
  \begin{equation}\label{eq:72}
    L(\hat{g}_t^{k,s}) < V_{p,g}^{E} \sqrt{n \Vol(E, g)}.
  \end{equation}
  For $n \leq 3$, $\frac{n}{2} - 2 < 0$ and therefore one can
  compute that $f_{k,s} \geq s > 0$ implies
  \begin{equation*}
    ((1 - t) + t f_{k,s})^{\frac{n}{2} - 2} < (1 - t)^{\frac{n}{2}
      - 2}.
  \end{equation*}
  In this case, then,
  \begin{equation}\label{eq:71}
    L(\hat{g}_t^{k,s}) < V_{p,g}^{E} \sqrt{n \Vol(E, g)} \integral{0}{1}{(1 -
      t)^{\frac{n}{4} - 1}}{dt} = V_{p,g}^{E} \sqrt{\Vol(E, g)} \cdot \frac{4}{\sqrt{n}}.
  \end{equation}
  Putting together \eqref{eq:72} and \eqref{eq:71} therefore gives
  \begin{equation}\label{eq:75}
    L(\hat{g}^{k,s}_t) \leq C(n) V_{p,g}^{E} \sqrt{\Vol(E, g)},
  \end{equation}
  where $C(n)$ is a constant depending only on $n$.

  In exact analogy, we can show that the same estimate holds for
  $\tilde{g}^{k,s}$ with $h$ in place of $g$.
  
  Next, we look at the second piece of $g^{k,s}$.  Here we have,
  using that $h - g = 0$ on $M \setminus E$,
  \begin{equation*}
    \begin{aligned}
      \abs{\bar{g}^{k,s}_t}_s^2 &= V_{f_{k,s} ((1 - t) g + t
            h)}^{-p} \integral{M}{}{
        \tr_{f_{k,s} ((1 - t) g + t h)} \left( (f_{k,s} (h -
          g))^2\right)}{dV_{f_{k,s} ((1 - t) g + t
            h)}} \\
      &= V_{f_{k,s} ((1 - t) g + t
            h)}^{-p} \integral{E}{}{f_{k,s}^{n/2} \tr_{(1 - t) g + t h} \left(
          (h - g)^2 \right)}{dV_{(1 - t) g + t
            h}}.
    \end{aligned}
  \end{equation*}
  Since $f_{k,s}(x) \leq 1$ for all $x \in M$
      it follows that $
      V_{f_{k,s} ((1 - t) g + t h)} \leq V_{(1-t) g + t h}$.
  Additionally, since $f_{k,s}(x) = s > 0$ for all $x \in M$ and
  $f_{k,s} \equiv 1$ on $E$, we have
  $V_{f_{k,s} ((1 - t) g + t h)} > \Vol(M \setminus E, (1 - t) g
  + t h)$.  Thus, defining (for $A \subseteq M$ measurable)
  \begin{equation*}
    W_{p,g,h}^A := \max
    \left\{
      V_{(1-t) g + t h}^{-p}, \Vol(M \setminus A, (1-t) g + t h)^{-p} \,
      : \, t \in [0,1]
    \right\},
  \end{equation*}
  the above estimate becomes
  \begin{equation}\label{BargEst}
    \begin{aligned}
      \abs{\bar{g}^{k,s}_t}_E^2 \leq s^{n/2} W_{p,g, h}^{E}
      \integral{F_k}{}{\tr_{(1 - t) g + t h} \left( (h - g)^2
        \right)}{&dV_{(1 - t) g + t
            h}} \\
                + W_{p,g, h}^{E} \integral{E \setminus F_k}{}{\tr_{(1 - t) g + t h}
        \left( (h - g)^2 \right)}{&dV_{(1 - t) g + t h}}.
    \end{aligned}
  \end{equation}
  For each fixed $t$, one can see that the first term in the above
  goes to
                zero as $k \rightarrow \infty$ followed by $s \rightarrow 0$.
  Additionally, by our assumption on the sets $F_k$, the second term
  in \eqref{BargEst} goes to zero as $k \rightarrow \infty$ for each
  fixed $t$ (it does not depend on $s$ at all).  Since $t$ only ranges
  over the compact interval $[0,1]$ and all terms in the integrals
  depend smoothly on $t$, both of these convergences are uniform in
  $t$.  From this, 
  \begin{equation}\label{eq:135}
    \lim_{s \rightarrow 0} \lim_{k \rightarrow \infty}
    L(\bar{g}^{k,s}) = 0.
  \end{equation}
  Combining \eqref{eq:75}, its analogue for $\tilde{g}^{k,s}$, and
  \eqref{eq:135}, together with 
$\lim_{k
    \rightarrow \infty} V_{p,g}^{E} = V_{p,g}^E$ (and similarly for $V_{p,h}^{E}$), gives
  the desired estimate.
\end{proof}

\begin{proof}[Proof of Proposition \ref{VolUpperBdpNE1}]
  The proof is divided into three cases: $p \leq 0$, $0 < p < 1$, and $p
  > 1$.

  First, let $p \leq 0$.  In this case, the result follows from
  Proposition \ref{VolUpperBd}, since $\max \{ V_g^{-p/2}, \Vol(M
  \setminus E, g)^{-p/2} \} = V_g^{-p/2}$, and similarly for $h$.

  Now, let $0 < p < 1$.  We use the notation of the proof of
  Proposition \ref{VolUpperBd}, and continue from \eqref{HatgEst}.
  Note that, since $p > 0$ and $f_{k,s} \geq s$,
  \begin{equation}\label{VhatgEst}
    V_{\hat{g}^{k,s}(t)}^{-p} =
    \left(
      \integral{M}{}{((1-t) + t f_{k,s})^{n/2}}{dV_g}
    \right)^{-p} \leq (1 - (1-s) t)^{-p n / 2} V_g^{-p}.
  \end{equation}

  Assume $n \leq 3$.  Then $\frac{n}{2} - 2 < 0$, and therefore
  \begin{equation}\label{fksEst}
    ((1-t) + t f_{k,s})^{\frac{n}{2} - 2} \leq (1 - (1-s) t)^{\frac{n}{2} - 2}.
  \end{equation}
  Also, $(1 - f_{k,s}) \leq (1-s)$, so combining this with
  \eqref{VhatgEst} and \eqref{fksEst} allows us to transform
  \eqref{HatgEst} into the estimate (with $\tau := (1-s) t$)
  \begin{equation}\label{LhatgEstnLT3}
    \begin{aligned}
      L(\hat{g}^{k,s}) &\leq \integral{0}{1}{\left( V_g^{-p}
          \integral{E}{}{(1 - (1-s) t)^{\frac{(1-p)n}{2} - 2} \tr_{g}
            \left( ((1 - s)g)^2 \right)
                      }{dV_{g}}
        \right)^{1/2}}{d t}\\
      &= V_g^{-p /2} \sqrt{n \Vol(E, g)}
      \integral{0}{1}{(1 - (1-s) t)^{\frac{(1-p)n}{4} - 1}}{ (1-s) dt} \\
      &= V_g^{-p /2} \sqrt{n \Vol(E, g)} \integral{0}{1-s}{(1 -
        \tau)^{\frac{(1-p)n}{4} - 1}}{d \tau} \\
      &\leq V_g^{-p /2} \sqrt{n \Vol(E, g)} \integral{0}{1}{(1 -
        \tau)^{\frac{(1-p)n}{4} - 1}}{d \tau} \\
      &\leq C(p, n) V_g^{-p /2} \sqrt{\Vol(E, g)},
    \end{aligned}
  \end{equation}
  where the last line follows since $p < 1$ and $n \leq 3$.
            
  Now, assume $n \geq 4$.  On $F_k$, we have $f_{k,s} \equiv s$, so we
  may carry out the same estimate as above (which, at least on $F_k$, does not depend
  on \eqref{fksEst}) to obtain
  \begin{equation}\label{LhatgEstnGT4}
    \begin{aligned}
      L(\hat{g}^{k,s}) &\leq C(p, n) V_g^{-p /2} \sqrt{\Vol(F_k, g)}
    \\
    &\qquad {} + \integral{0}{1}{\left( V_{\hat{g}^{k,s}(t)}^{-p}
        \integral{E \setminus F_k}{}{((1 - t) + t
          f_{k,s})^{\frac{n}{2} - 2} \tr_{g} \left( ((1 -
            f_{k,s})g)^2 \right)
        }{dV_{g}}
      \right)^{1/2}}{d t}.
    \end{aligned}
  \end{equation}
  Since, in this case, $\frac{n}{2} - 2 \geq 0$, the fact that
  $f_{k,s} \leq 1$ implies $((1 - t) + t f_{k,s})^{\frac{n}{2} - 2}
  \leq 1$.  Also, since $f_{k,s} > 0$, we have that $1 - f_{k,s} < 1$.
  Using these facts, together with \eqref{VhatgEst}, The second term
  on the right-hand side of the above expression can be estimated from
  above by
  \begin{equation*}
    V_g^{-p/2} \sqrt{n \Vol(E \setminus F_k, g)} \integral{0}{1}{(1 -
      (1-s) t)^{-pn/2}}{d t}.
  \end{equation*}
  The value of the integral in the above is finite for each fixed $s >
  0$ and does not depend on $k$.  Furthermore, by our assumptions on
  the sets $E$ and $F_k$, the above expression goes to zero as $k
  \rightarrow \infty$.  Combining this fact with \eqref{LhatgEstnLT3}
  and \eqref{LhatgEstnGT4} shows that for any $n$,
  \begin{equation*}
    \lim_{k \rightarrow \infty} L(\hat{g}^{k,s}) \leq C(p, n)
    V_g^{-p/2} \sqrt{\Vol(E, g)}. 
  \end{equation*}
  A similar estimate holds for $L(\tilde{g}^{k,s})$, and we can show
  exactly as in the proof of Proposition \ref{VolUpperBd} that
  $\lim_{s \rightarrow 0} \lim_{k \rightarrow \infty} L(\bar{g}^{k,s})
  = 0$.  This completes the proof for $0 < p < 1$.

  Finally, let $p > 1$.  In this case, we use the isometry $F$ from
  Proposition \ref{IsometryProp} and the result for $p < 1$ to see
  \begin{equation*}
    \begin{aligned}
      d_p(g, h)       \leq C(p, n) \cdot \left(
        V_{F(g)}^{\frac{p-2}{2}} \sqrt{\Vol(E, F(g))} + V_{F(h)}^{\frac{p-2}{2}} \sqrt{\Vol(E,F(h))}
      \right).
    \end{aligned}
  \end{equation*}
  Recalling that $V_{F(g)} = V_g^{-1}$ and noting that $\Vol(E, F(g))
  = V^{-2} \Vol(E, g)$ (and similarly for $F(h)$) then leads to the
  result.
\end{proof}

\end{document}